\documentclass[a4paper, 11pt]{article}

\usepackage[margin=1in]{geometry}
\usepackage{setspace}
\usepackage{microtype}

\usepackage{kotex}

\usepackage{amsmath} \allowdisplaybreaks
\usepackage{amssymb}
\usepackage{amsthm} \theoremstyle{definition}
\newtheorem{theorem}{Theorem}
\newtheorem{lemma}{Lemma}

\newtheorem{assumption}{Assumption}
\newtheorem{proposition}{Proposition}
\newtheorem{definition}{Definition}

\usepackage{natbib}
\usepackage{bibentry}

\usepackage{booktabs}
\usepackage{graphicx}
\usepackage{appendix}

\usepackage{subcaption}

\usepackage[hyphens]{url}
\usepackage[bookmarks=false,hidelinks]{hyperref}

\usepackage{multirow}

\usepackage{tikz}
\usepackage{pgfplots}
\usetikzlibrary{positioning}
\usetikzlibrary{arrows,shapes}
\usetikzlibrary{snakes}

\usepackage{authblk}

\usepackage[shortlabels]{enumitem}

\usepackage{pdflscape}

\newcommand{\email}[1]{{\href{mailto:#1}{\nolinkurl{#1}}}}

\usepackage{bm}

\usepackage{xcolor}

\usepackage{etoolbox}
\makeatletter
\def\do#1{\@namedef{#1c}{\ensuremath{\mathcal{#1}}}}
\docsvlist{A,B,C,D,E,F,G,H,I,J,K,L,M,N,O,P,Q,R,S,T,U,V,W,X,Y,Z}
\makeatother

\def\Eb{\mathbb{E}}
\def\Rb{\mathbb{R}}

\newcommand{\dz}{\mathop{} \mathrm{d}z}

\newcommand{\du}{\mathop{} \mathrm{d}u}
\newcommand{\dv}{\mathop{} \mathrm{d}v}

\renewcommand{\hat}{\widehat}

\DeclareMathOperator*{\Var}{Var}
\DeclareMathOperator*{\diam}{diam}

\def\TSP{\mathsf{TSP}}
\def\TSPD{\mathsf{TSPD}}
\def\kTSP{k\mathsf{TSP}}

\newcommand{\meanstd}[2]{#1{\raisebox{-0.25ex}{\tiny$\pm#2$}}}
\newcommand{\meanonly}[1]{#1{\raisebox{-0.25ex}{\tiny\phantom{$\pm0.0000$}}}}

\def\Binomial{\text{Binomial}}
\def\Exponential{\text{Exponential}}
\def\Poisson{\text{Poisson}}
\def\Uniform{\text{Uniform}}

\title{Asymptotic Bounds for the Traveling Salesman Problem with Drone}
\author[a]{Jae Hyeok Lee\thanks{이재혁}}
\author[a]{Taekang Hwang\thanks{황태강}}
\author[a, b]{Changhyun Kwon\thanks{권창현; Corresponding author: \email{chkwon@kaist.ac.kr}}}
\affil[a]{Department of Industrial and Systems Engineering, KAIST, Daejeon, 34141, Republic of Korea}
\affil[b]{Omelet, Inc., Daejeon, 34051, Republic of Korea}
\date{August 4, 2026}

\begin{document}
\onehalfspacing
\maketitle

\begin{abstract}
The asymptotic behavior of the optimal TSP tour length is well known from the classical Beardwood--Halton--Hammersley theorem. We extend this result to the Traveling Salesman Problem with Drone (TSPD), a cooperative routing problem in which a truck and a drone jointly serve customers. Using a nonmonotone subadditive Euclidean functional framework, we establish the existence of an almost sure limit for the optimal TSPD makespan scaled by the square root of the problem size. We derive explicit upper and lower bounds for the speed-scaled Euclidean TSPD model: upper bounds are obtained via structured ring-based tour constructions and Monte Carlo evaluation, while lower bounds are derived using nearest-neighbor distance distributions and the $k$-traveling salesman problem. Computational results illustrate how tight the bounds are. We also extend the analysis to the Rectilinear–Euclidean mixed TSPD model, in which truck travel is measured by the rectilinear distance and drone travel by the Euclidean distance.
\end{abstract}

\section{Introduction}

The seminal work of \citet*[BHH][]{beardwood1959shortest} showed that the length of the optimal TSP tour converges as the problem size increases. 
While the original result is more general, we focus on the case of the unit square in $\Rb^2$.

\begin{theorem}[BHH]
Let $X_1, X_2, \ldots, X_n$ be a set of $n$ points, uniformly distributed on the unit box $[0, 1]^2$, and let $\TSP(\{X_1, X_2, \ldots, X_n\})$ be the length of the optimal TSP tour. Then,
\begin{equation}
\lim_{n \to \infty} \frac{\TSP(\{X_1, X_2, \ldots, X_n\})}{\sqrt{n}} = \beta_\TSP
\end{equation}
for some positive constant $\beta_\TSP$, almost surely.
\end{theorem}

While the exact value of $\beta_\TSP$ is still unknown, empirical studies suggest that $\beta_\TSP \approx 0.71$ \citep{johnson1996asymptotic, choi2022review}.
Lower and upper bounds are known as follows:
\begin{equation}
0.6277 \leq \beta_\TSP < 0.9038
\end{equation}
where the lower bound is by \citet{gaudio2020improved} and the upper bound is by \citet{carlsson2025new}.

The Beardwood–Halton–Hammersley (BHH) theorem has been widely used and extended in continuous approximation models in transportation \citep{daganzo1984distance,daganzo1984length}. 
Applications include location problems \citep{carlsson2022continuous}, the location-or-routing problem \citep{arslan2021location}, shared electric scooter rebalancing \citep{osorio2021optimal}, bike-sharing market expansion \citep{fu2022bike}, and same-day delivery systems \citep{banerjee2023has}. 
Readers are referred to the review articles by \citet{langevin1996continuous} and \citet{ansari2018advancements} for historical perspectives and a broader overview of applications.

In this paper, we aim to establish the existence of a BHH-type limit and derive explicit upper and lower bounds for a collaborative delivery problem known as the Traveling Salesman Problem with Drone (TSPD) defined in two-dimensional space \citep{murray2015flying, agatz2018optimization}.
The TSPD utilizes a truck and a drone to fulfill the delivery needs, while allowing the drone to launch from and land on the truck. 
Establishing asymptotic bounds for TSPD is crucial for two reasons.
First, it enables rapid, formula-based cost estimation for strategic logistics systems design as demonstrated in the continuous-approximation literature.
Second, it provides a rigorous baseline for benchmarking large-scale heuristics, metaheuristics, or deep learning methods, where exact solutions are unattainable.

This study is closely related to \citet{carlsson2018coordinated} and \citet{jones2026coordinated}, which analyze the asymptotic behavior of coordinated truck-drone delivery systems.
In their models, the drone launches from and rendezvouses with the truck at optimally chosen \emph{continuous} points along the truck path.
In contrast, the TSPD restricts launching and landing to \emph{discrete} points---the customer nodes.
This discreteness introduces additional synchronization constraints while reflecting practical constraints imposed by the truck's road network.
Analyzing the characteristics of the connections among those discrete points, we establish a BHH-type asymptotic limit and derive lower and upper bounds for the coordinated delivery.
For \citet{carlsson2018coordinated}, the analytical bounds are derived under a continuous approximation in which customer locations are represented by a spatial density.
Consequently, those bounds are not directly comparable to the bounds developed in this paper.
Nevertheless, their computational results provide an empirical reference for assessing the cost of discreteness.
\citet{jones2026coordinated} study a multiple-sidekick model, whose single-sidekick case falls within the continuous model described above.
While their lower-bound arguments remain valid for TSPD, their upper-bound constructions use continuous rendezvous locations and therefore do not yield feasible TSPD tours.

Our contributions are threefold.
First, we prove the almost sure convergence of the TSPD makespan using a nonmonotone subadditive Euclidean functional framework.
Second, we derive upper bounds by analyzing the properties of optimal TSPD solutions using a structured ring-based tour-construction framework. 
Most importantly, we prove that there always exists an optimal TSPD solution where the truck and drone never travel together.
This structural property allows us to rule out many candidate configurations and substantially reduce the computational effort required to evaluate upper bounds. 
Third, we derive lower bounds by leveraging nearest-neighbor distance distributions and the $k$-traveling salesman problem ($k$-TSP).
The resulting bound is tighter than existing lower bounds.
Our approach applies broadly to models in which the truck distance is measured by different metrics, such as the rectilinear and Euclidean norms.

The rest of this paper is organized as follows.
Section~\ref{sec_prelim} summarizes the results on the TSP.
Section~\ref{sec_tspd} introduces the TSPD and proves the existence of its asymptotic limit.
Sections~\ref{sec_upper_bounds} and~\ref{sec_lower_bounds} derive upper and lower bounds, respectively.
Section~\ref{sec_results} presents the numerical results.
Section~\ref{sec_mixed} extends the analysis to the Rectilinear-Euclidean mixed model, and
Section~\ref{sec_conclusion} concludes the paper.

\paragraph{Notation.} Throughout the paper, let $x_i\in\Rb^2$ denote a deterministic customer location and $\Xc=\{x_1,\ldots,x_n\}$ a finite deterministic customer set.
We use $X_i$ to denote a random customer location.
We use $\mathcal{L}$ for a generic real-valued functional on finite point sets in $\Rb^2$.
In particular, $\TSPD(\Xc)$ denotes the optimal TSPD makespan on $\Xc$.
The corresponding asymptotic constant is $\beta_\TSPD$.
When the drone speed factor $\alpha\geq1$ is made explicit, we write $\beta_\TSPD(\alpha)$.
Let $\|\cdot\|_T$ and $\|\cdot\|_D$ denote the truck and drone metrics,
respectively; both are norms on $\Rb^2$.
Unless otherwise stated, we assume that the drone metric is no larger
than the truck metric; that is,
\[
	\|p\|_D\leq \|p\|_T
	\qquad \text{for all } p\in\Rb^2.
\]
Finally, for a rectangle $\Qc\subset\Rb^2$, $\diam\Qc$ denotes the Euclidean length of its diagonal, equivalently
\[
	\diam\Qc=\sup\{\|x-y\|_2:x,y\in\Qc\}.
\]

\section{Preliminaries} \label{sec_prelim}

In this section, we provide a generalized result for the existence of the limit and summarize the approaches for bounding the limit in the TSP.
These results will provide a basis for the TSPD analysis later.

\subsection{The Existence of the Limit}

\citet{steele1981subadditive} generalized the approach of \citet{beardwood1959shortest} for more general settings to prove the existence of the limit and provided an axiomatic approach.
\begin{assumption}[\citealt{steele1981subadditive}] \label{assumption_steele}
We make the following assumptions on the function $\mathcal{L}(\cdot)$:
\begin{enumerate}[({A}1)]
\item \emph{Positive Homogeneity.} $\mathcal{L}(\{\lambda x_1, \ldots, \lambda x_n\}) = \lambda \mathcal{L}(\{x_1, \ldots, x_n\}) $ for all $\lambda>0$.

\item \emph{Translational Invariance.} $\mathcal{L}(\{x_1 + x, \ldots, x_n + x\}) = \mathcal{L}(\{x_1, \ldots, x_n\}) $ for any $x \in \Rb^2$.

\item \emph{Monotonicity.} $\mathcal{L}(\{x\} \cup \Sc) \geq \mathcal{L}(\Sc)$ for any $x \in \Rb^2$ and finite subset $\Sc \subset \Rb^2$.

\item \emph{Finite Variance.} $\Var(\mathcal{L}(\{X_1, \ldots, X_n\})) < \infty$ whenever $X_1, \ldots, X_n$ are independent and uniformly distributed on $[0, 1]^2$.

\item \emph{Subadditivity.} 
For any positive integer $m$, let $\{\Qc_i : i = 1, 2, \ldots, m^2\}$ be a partition of $[0, 1]^2$ into squares of length $1/m$. 
Let $\Xc = \{x_1, \ldots, x_n\}$, and let $t \Qc_i = \{ t w : w \in \Qc_i\}$ for any $t>0$.  
There exists a constant $C > 0$ such that 
\[
	\mathcal{L}(\Xc \cap [0, t]^2) \leq
	\sum_{i=1}^{m^2} \mathcal{L}(\Xc \cap t \Qc_i) + C t m.
\]
\end{enumerate}
\end{assumption}

Assumptions (A1) and (A2) indicate that the function $\mathcal{L}(\cdot)$ is a Euclidean functional, of which $\TSP(\cdot)$ is a clear example.
We note, however, that Euclidean functionals from some combinatorial optimization problems are not monotone; for example, the minimum spanning tree problem.
For such problems, \citet{rhee1993limit} generalized the approach of \citet{steele1981subadditive} to prove the existence of the limit under conditions that do not require monotonicity.
\begin{assumption}[\citealt{rhee1993limit}] \label{assumption_rhee}
We make the following assumptions on the function $\mathcal{L}(\cdot)$:
\begin{description}
\item[\rm (A0)] $\mathcal{L}(\emptyset)=0$.

\item[\rm (A1) \it Positive Homogeneity.] $\mathcal{L}(\{\lambda x_1, \ldots, \lambda x_n\}) = \lambda \mathcal{L}(\{x_1, \ldots, x_n\})$ for all $\lambda>0$.

\item[\rm (A2) \it Translational Invariance.] $\mathcal{L}(\{x_1 + x, \ldots, x_n + x\}) = \mathcal{L}(\{x_1, \ldots, x_n\})$ for any $x \in \Rb^2$.

\item[\rm (A5$'$) \it Geometric Subadditivity.]
For $m=2$ or $m=3$, let $\{\Qc_i : i = 1, 2, \ldots, m^2\}$ be a partition of $[0,1]^2$ into squares of side length $1/m$.
There exists a constant $C'>0$ such that, for every finite set $\Xc \subset [0,1]^2$ that does not meet the boundaries of the partition,
\[
	\mathcal{L}(\Xc) \leq
	\sum_{i=1}^{m^2} \mathcal{L}(\Xc \cap \Qc_i) + C' .
\]

\item[\rm (B1) \it Smoothness.]
There exists a constant $K > 0$ such that, for all finite sets $\Fc,\Gc \subset [0,1]^2$,
\[
	|\mathcal{L}(\Fc \cup \Gc)-\mathcal{L}(\Fc)|
	\leq K \sqrt{|\Gc|}.
\]
\end{description}
\end{assumption}

The smoothness condition (B1) replaces the monotonicity condition used in the approach of \citet{steele1981subadditive} and applies to nonmonotone Euclidean functionals such as the matching problem \citep{rhee1993limit}.

The following theorem is a generalized result:

\begin{theorem}[\citealt{rhee1993limit}] \label{thm_rhee}
Let $X_1, X_2, \ldots, X_n$ be a set of $n$ points, uniformly distributed on the unit box $[0, 1]^2$.
Suppose all conditions in Assumption \ref{assumption_rhee} hold.
There exists a constant $\beta(\mathcal{L})$ such that
\begin{equation}
	\lim_{n \to \infty} \frac{\mathcal{L}(\{X_1, \ldots, X_n\})}{\sqrt{n}} = \beta(\mathcal{L})
\end{equation}
almost surely.
\end{theorem}

The subadditivity condition in (A5$'$) is the most crucial part in guaranteeing the existence of the limit.
Later, \citet{yukich2006probability} provided a simpler sufficient form of subadditivity:
\begin{description}
\item[\rm (A5$''$) \it Rectangle Form of Geometric Subadditivity.]
Let $\Xc = \{x_1, \ldots, x_n\}$.
If a rectangle $\Qc$ is divided into two rectangles $\Qc_1$ and $\Qc_2$, there exists $C'' > 0$ such that
\[
	\mathcal{L}( \Xc \cap \Qc )
	\leq
	\mathcal{L}( \Xc \cap \Qc_1 )
	+
	\mathcal{L}( \Xc \cap \Qc_2 )
	+
	C'' \diam\Qc.
\]
\end{description}

Repeated application of (A5$''$) implies (A5$'$) for the $m=2$ and $m=3$ partitions of the unit square.

\subsection{Poissonization}

Consider $n$ points, $X_1, \ldots, X_n$, uniformly distributed in the unit box $[0, 1]^2$.
Due to the following lemma, we can focus on bounding the expected value instead of dealing with the `almost surely' condition:
\begin{lemma}[\citealt{beardwood1959shortest}]
\begin{equation}
\lim_{n \to \infty} \frac{\Eb[\TSP(\{X_1, X_2, \ldots, X_n\})]}{\sqrt{n}} = \beta_\TSP.
\end{equation}
\end{lemma}

Further, we notice that considering $n$ points in $[0, 1]^2$ is the same as considering the Poisson process with intensity $n$ \citep{beardwood1959shortest}.
Using this observation, we can find an upper bound on $\beta_\TSP$.

\begin{figure}
\centering
\begin{subfigure}[t]{0.48\textwidth}
    \centering
    \includegraphics[width=0.8\linewidth]{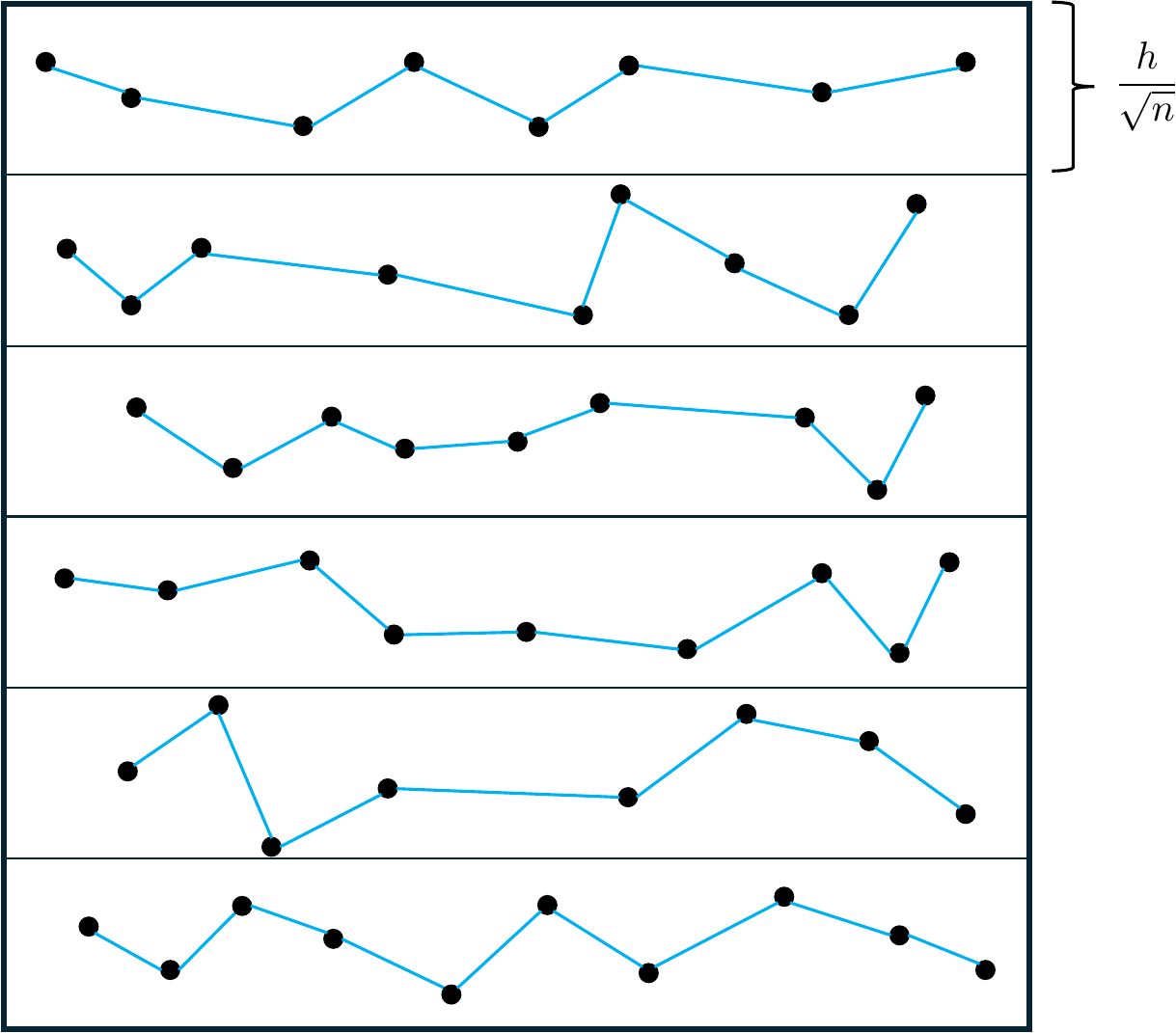}
    \caption{Divide the unit square $[0,1]^2$ by strips with equal height of $h/\sqrt{n}$ and connect points within each strip left to right.}
    \label{fig_tsp1}
\end{subfigure}\hfill
\begin{subfigure}[t]{0.48\textwidth}
    \centering
    \includegraphics[width=0.8\linewidth]{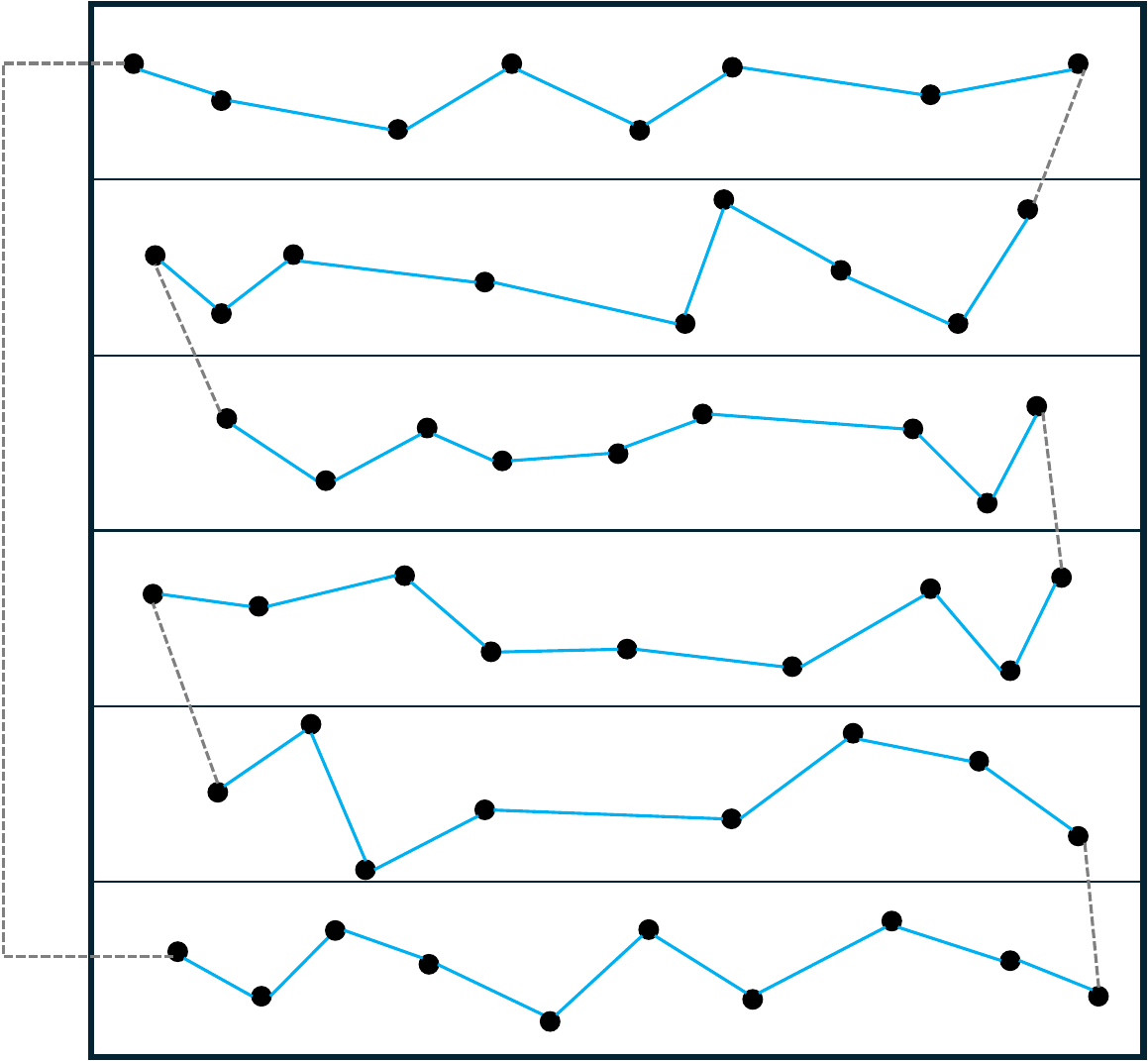}
    \caption{To complete a TSP tour, connect paths by extra lines, shown in dotted lines, whose total length is $o(\sqrt{n})$ almost surely.}
    \label{fig_tsp2}
\end{subfigure}

\caption{Creating a feasible TSP tour that passes through $n$ points in $[0,1]^2$. [Reproduced from Figure 1 of \citet{carlsson2025new}]}
\label{fig_tspd_comparison}
\end{figure}

Consider strips of the box with height $h / \sqrt{n}$ for some constant $h > 0$, as shown in Figure~\ref{fig_tspd_comparison}.
There are $\sqrt{n}/h$ strips.
In each strip, the number of points is $\frac{n}{\sqrt{n}/h} = h\sqrt{n}$ on average.
Let $m = h\sqrt{n}$.
In each strip, we will connect the points by straight lines in the order of their horizontal coordinates. 

Consider the horizontal coordinates first.
For any arbitrary interval of length $\delta$, a point lies inside this interval with probability $\delta$ and outside with probability $1-\delta$.
Therefore, we can consider this as a Bernoulli trial. 
The number of points in this interval is distributed as $\Binomial(m, \delta)$.
When $m$ is large and $\delta$ is small, it is well known that the Binomial distribution behaves like $\Poisson(m\delta)$.
Since intervals of length $\delta$ are independent of each other, we can now consider the entire interval $[0, 1]$ as a Poisson point process with intensity $m$. 

This implies that the distance between two consecutive horizontal coordinates is distributed as $\Exponential(m)$.
Let us denote the horizontal coordinate of point $X_i$ by $X^{(1)}_i$.
Considering two arbitrary consecutive points $X_i$ and $X_{i+1}$, we have:
\[
	X^{(1)}_{i+1} - X^{(1)}_i \ \sim \ \Exponential(h\sqrt{n}),
\]
or, in the \emph{unscaled} form,
\[
	Z \ \sim \ \Exponential(1),
\]
where $Z = h\sqrt{n} (X^{(1)}_{i+1} - X^{(1)}_i)$.

Similarly, let us consider the vertical coordinates, denoted by $X^{(2)}_i$. 
In each strip of height $h/\sqrt{n}$, it is clear that $X^{(2)}_i = \frac{h}{\sqrt{n}} U$, where 
\[
	U \ \sim \ \Uniform([0, 1]).
\]

\subsection{Upper Bounds}

Using these observations, we find an upper bound on $\beta_\TSP$.
First, we note that we do not need to consider the connections between the strips.
Since there are $\sqrt{n}/h$ number of strips, the total length of those connections is $o(\sqrt{n})$; therefore, it does not impact $\beta$.

For two arbitrary consecutive points $(X^{(1)}_i, X^{(2)}_i)$ and $(X^{(1)}_{i+1}, X^{(2)}_{i+1})$:
\[
	\Eb \left\lVert \begin{pmatrix} 
		X^{(1)}_{i+1} - X^{(1)}_i \\ 
		X^{(2)}_{i+1} - X^{(2)}_i 
	\end{pmatrix}\right\rVert 
	= \Eb \left\lVert \begin{pmatrix} \frac{1}{h\sqrt{n}} Z \\ \frac{h}{\sqrt{n}} (U_1 - U_0) \end{pmatrix} \right\rVert  
	= \frac{1}{h\sqrt{n}} \Eb \left\lVert \begin{pmatrix} Z \\  h^2(U_1 - U_0) \end{pmatrix} \right\rVert.
\]
Recall that there are $\sqrt{n}/h$ number of strips, and, in each strip, the number of points on average is $h\sqrt{n}$.
Therefore, we obtain, for large $n$,
\begin{align*}
\frac{\Eb[\mathsf{TSP}(\{X_1, X_2, \ldots, X_n\})]}{\sqrt{n}}
&\leq \frac{1}{\sqrt{n}} \cdot \frac{\sqrt{n}}{h} \cdot (h\sqrt{n}) \cdot \frac{1}{h\sqrt{n}} \Eb \left\lVert \begin{pmatrix} Z \\ h^2(U_1 - U_0) \end{pmatrix} \right\rVert \\
&= \frac{1}{h} \Eb \left\lVert \begin{pmatrix} Z \\ h^2(U_1 - U_0) \end{pmatrix} \right\rVert \\
&= \frac{1}{h} \int_0^\infty \int_0^1 \int_0^1 e^{-z} \sqrt{z^2 + h^4(u-v)^2} \dv \du \dz.
\end{align*}
We can write
\[
\beta_\TSP \leq \inf_{h > 0} \frac{1}{h} \int_0^\infty \int_0^1 \int_0^1 e^{-z} \sqrt{z^2 + h^4(u-v)^2} \dv \du \dz,
\]
which is minimized near $h=\sqrt{3}$.
The upper bound can be numerically evaluated to be 0.92116.
\citet{steinerberger2015new} later improved this bound as
\[
	\beta_\TSP \leq  0.92116 - \frac{9}{16} 10^{-6} \approx 0.9211594,
\]
and \citet{carlsson2025new} improved further as
\[
	\beta_\TSP < 0.9038.
\]

\subsection{Lower Bounds}
\label{sec_tsp_lower_bounds}

To find a lower bound, \citet{beardwood1959shortest} utilized the fact that each point is connected to two other points. 
If we assume that each point is connected to the nearest and the second-nearest points, then the total length of such connections will provide us with a lower bound. 
Of course, it may yield disconnected components, resulting in an infeasible solution.
The approach is well described in \citet{steinerberger2015new}.

\begin{lemma}[\citealt{steinerberger2015new}]
\label{lem_nearest}
Let $\Pc_n$ be a Poisson process on $\Rb^2$ with intensity $n$. %
Then, for any fixed point $p \in \Rb^2$, the probability distribution of the distance between $p$ and the nearest point in $\Pc_n$ is given by $f_1(r) = 2\pi n r e^{-\pi n r^2}$ and its expectation is $\frac{1}{2\sqrt{n}}$.
The distance to the second-nearest point is distributed by $f_2(r) = 2\pi^2 n^2 r^3 e^{-\pi n r^2}$ and its expectation is $\frac{3}{4\sqrt{n}}$.
\end{lemma}

Consequently, Lemma~\ref{lem_nearest} yields that
\[
	\TSP(\{X_1, \ldots, X_n\}) \geq \bigg(\frac{1}{2} + \frac{3}{4}\bigg) \frac{1}{\sqrt{n}} \frac{n}{2} = \frac{5}{8} \sqrt{n};
\]
hence,
\[
	\beta_\TSP \geq \frac{5}{8},
\]
which was the result of \citet{beardwood1959shortest}.
Later \citet{steinerberger2015new} and \citet{gaudio2020improved}, respectively, improved the lower bound to 
\[
	\beta_\TSP \geq \frac{5}{8} + \frac{19}{10368} > 0.6268.
\]
and
\[
	\beta_\TSP \geq \frac{5}{8} + \frac{1}{2} \bigg(\frac{19}{10368} \bigg)
	+ \frac{1}{2}\bigg(
		\frac{3072\sqrt{2} - 4325}{5376}
	\bigg)
	> 0.6277.
\]

\section{The TSPD Case}
\label{sec_tspd}

In this section, we first provide a formal definition of the TSPD and then show that the limit exists for the TSPD.

\subsection{The TSPD Defined}
For any finite point set $\Xc\subseteq\Rb^2$, the goal of the TSPD problem is to find coordinated truck and drone routes that minimize the total makespan of operations.
We distinguish node types: \emph{truck nodes} are visited only by the truck, \emph{drone nodes} are served only by the drone, and \emph{combined nodes} are synchronization points for both vehicles.

\begin{figure}
\begin{subfigure}[t]{0.23\textwidth}
    \centering
    \includegraphics[scale=0.5]{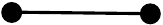}
    \caption{A straight ring}
    \label{fig_a_straight_ring}
\end{subfigure}\hfill
\begin{subfigure}[t]{0.23\textwidth}
    \centering
    \includegraphics[scale=0.5]{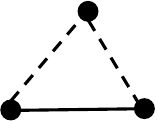}
    \caption{A triangle ring}
    \label{fig_a_triangle_ring}
\end{subfigure}\hfill
\begin{subfigure}[t]{0.23\textwidth}
    \centering
    \includegraphics[scale=0.5]{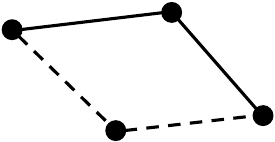}
    \caption{A quartet ring}
    \label{fig_a_quartet_ring}
\end{subfigure}\hfill
\begin{subfigure}[t]{0.23\textwidth}
    \centering
    \includegraphics[scale=0.5]{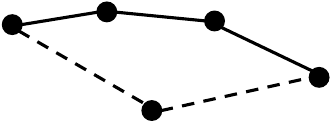}
    \caption{A quintet ring}
    \label{fig_a_quintet_ring}
\end{subfigure}\hfill
\caption{Rings. Solid lines are traveled by the truck, and dotted lines are traveled by the drone. In the straight ring, the truck travels with the drone attached.}
\label{fig_rings}
\end{figure}

We adopt the term \emph{rings} as the basic unit of TSPD solution constructions.
Examples of rings are shown in Figure~\ref{fig_rings}, and the definition of a ring is given below.

\begin{definition}
A \emph{ring} is defined by two distinct combined nodes serving as the start and end nodes, with at most one drone node and any number of truck nodes between them. 
For each ring $r$, the truck follows a path $\Tc_r$, and if a drone node is present, the drone follows a path $\Dc_r$.
\end{definition}

Define the truck and drone path costs as
\[
t(\Tc_r) = \sum_{(i,j)\in \Tc_r} \|x_i - x_j\|_T,
\qquad
d(\Dc_r) = \sum_{(i,j)\in \Dc_r} \|x_i - x_j\|_D.
\]
The cost of a ring $r$ is then:
\begin{align*}
c(r) 
& = c(\Xc_r) \\
& = \min_{\Tc_r, \Dc_r} c(\Tc_r, \Dc_r) \\
& = \min_{\Tc_r, \Dc_r} 
\begin{cases}
\max\{  t(\Tc_r),  d(\Dc_r) \}, & \text{if $r$ contains a drone node}, \\
t(\Tc_r), & \text{otherwise},
\end{cases}
\end{align*}
where we abuse notation by defining $\Xc_r$ as the set of points in ring $r$, with the first and last points being the two combined nodes.

A feasible solution to the TSPD problem is a set of rings $\Rc(\Xc)$ that together cover all points in $\Xc$. 
The rings in $\Rc(\Xc)$ are concatenated through the combined nodes in the TSPD solution to form a single cyclic tour.
Following the literature, in a feasible TSPD solution, we assume that each point in $\Xc$ is served exactly once, either by the truck, the drone, or both.
Hence, if a point in $\Xc$ is visited by both the truck and the drone, then it is a combined node shared by exactly two adjacent rings in $\Rc(\Xc)$.
Otherwise, the point belongs to exactly one ring.

The objective is to minimize the total cost:
\[
\TSPD(\Xc)  =  \min_{\Rc(\Xc)} \sum_{r\in\Rc(\Xc)} c(r),
\]
which is the makespan.

For degenerate cases, we use the following convention.
If $|\Xc|=0$ or $|\Xc|=1$, then $\TSPD(\Xc)=0$.
If $\Xc=\{x,y\}$, then $\TSPD(\Xc)=2\|x-y\|_T$, corresponding to the straight-ring tour between the two points.

\subsection{The Existence of the Limit}

We present the existence of the limit for the TSPD case.
Unlike the continuous model \citep{jones2026coordinated}, the TSPD functional need not be monotone because adding a customer can create a new synchronization point.
Figure~\ref{fig_tspd_nonmonotonicity} illustrates the following example.
Consider Euclidean truck distances and drone distances equal to half the Euclidean distances.
Let $a=(0,0)$, $b=(0,1)$, and $c=(1/2,0)$.
By the convention above, $\TSPD(\{a,b\})=2$.
For $\{a,b,c\}$, consider a triangle ring from $a$ to $c$ that serves $b$ by drone, followed by a straight ring from $c$ to $a$.
\mbox{This feasible solution has cost}
\nopagebreak[4]
\[
\max\left\{\frac{1}{2},\frac{1}{2}\big(\|a-b\|_2+\|b-c\|_2\big)\right\}+\frac{1}{2}
=1+\frac{\sqrt{5}}{4}<2.
\]
Thus, $\TSPD(\{a,b,c\})<\TSPD(\{a,b\})$.
We therefore verify Assumption~\ref{assumption_rhee}, which does not require monotonicity, and apply Theorem~\ref{thm_rhee}.

\begin{figure}[t]
\centering
\begin{subfigure}[t]{0.38\textwidth}
    \centering
    \begin{tikzpicture}[line width=0.8pt,
        point/.style={circle,fill=black,inner sep=2.6pt}]
        \node[point,label=below:$a$] (a) at (0,0) {};
        \node[point,label=left:$b$] (b) at (0,2) {};
        \draw (a) -- (b);
    \end{tikzpicture}
    \caption{The truck-only tour for $\{a,b\}$; the edge is traversed twice.}
    \label{fig_nonmonotone_ab}
\end{subfigure}\hfill
\begin{subfigure}[t]{0.48\textwidth}
    \centering
    \begin{tikzpicture}[line width=0.8pt,
        drone/.style={dash pattern=on 4pt off 3pt},
        point/.style={circle,fill=black,inner sep=2.6pt}]
        \node[point,label=below:$a$] (a) at (0,0) {};
        \node[point,label=below:$c$] (c) at (1,0) {};
        \node[point,label=left:$b$] (b) at (0,2) {};
        \draw (a) -- (c);
        \draw[drone] (a) -- (b) -- (c);
    \end{tikzpicture}
    \caption{A feasible solution for $\{a,b,c\}$; the truck traverses $ac$ twice.}
    \label{fig_nonmonotone_abc}
\end{subfigure}
\caption{A counterexample to monotonicity. Solid and dashed lines represent truck and drone travel, respectively.}
\label{fig_tspd_nonmonotonicity}
\end{figure}

To show that $\TSPD(\cdot)$ satisfies (A5$''$) and (B1) in Assumption~\ref{assumption_rhee}, we propose a procedure, illustrated in Figure~\ref{fig_tspd_subadditivity}, for combining two local TSPD tours on disjoint customer sets into a feasible tour for $\Xc\cap\Qc$.
The procedure consists of three steps: exposing straight rings, removing them, and connecting the resulting tours (see Figures~\ref{fig_tspd3} and~\ref{fig_tspd4}).
For the constructed solution $\Rc'$ shown in the figure, we obtain
\begin{align*}
	\TSPD(\Xc) 
	&= c(\Rc) \\
	&\leq c(\Rc') \\
	&= c(\Rc_1) + c(\Rc_2) + \Delta+ \|a-d\|_T + \|b-c\|_T- \|a-b\|_T - \|c-d\|_T \\
	&\leq c(\Rc_1) + c(\Rc_2) + \Delta+ \|a-d\|_T + \|b-c\|_T \\
	&= \TSPD(\Xc_1) + \TSPD(\Xc_2) + \Delta+ \|a-d\|_T + \|b-c\|_T,
\end{align*}
Here, $\Delta$ is the cost increase from exposing the straight rings $ab$ and $cd$.
If $r_1$ and $r_2$ denote the two chosen rings in Figure~\ref{fig_tspd2}, and $\hat r_1$ and $\hat r_2$ denote the corresponding residual rings after exposing $ab$ and $cd$, then it follows that
\[
\Delta=\bigl(c(\hat r_1)+\|a-b\|_T-c(r_1)\bigr)+\bigl(c(\hat r_2)+\|c-d\|_T-c(r_2)\bigr)\ge 0,
\]
where the last inequality follows from the triangle inequality and the drone-metric dominance condition.
Hence, for the given instance, the additional cost of the procedure is bounded by a constant.
Interestingly, this holds for all possible instances.
Indeed, the proof of Theorem~\ref{thm_tspd_limit} shows that the same type of bound holds with a constant independent of $\Xc$, and uses this fact directly to verify that (A5$''$) and (B1) hold for $\TSPD(\cdot)$.

\begin{figure}
\centering
\begin{subfigure}[t]{0.48\textwidth}
    \centering
    \includegraphics[width=0.9\linewidth]{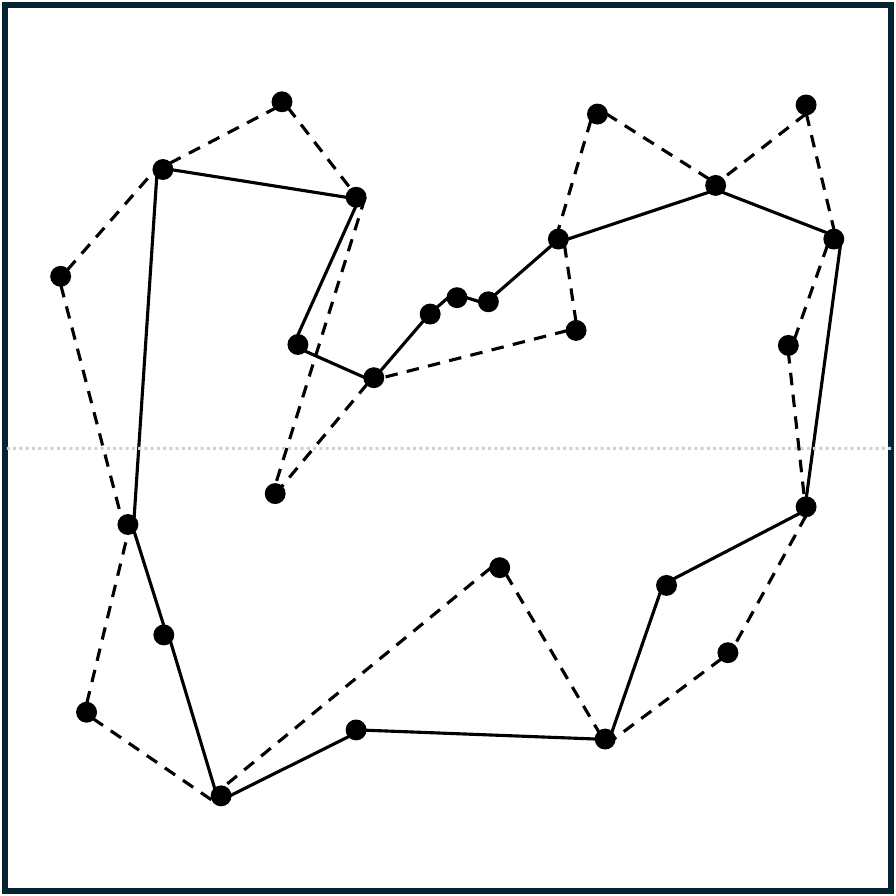}
    \caption{An optimal TSPD solution for $\Xc$.}
    \label{fig_tspd1}
\end{subfigure}\hfill
\begin{subfigure}[t]{0.48\textwidth}
    \centering
    \includegraphics[width=0.9\linewidth]{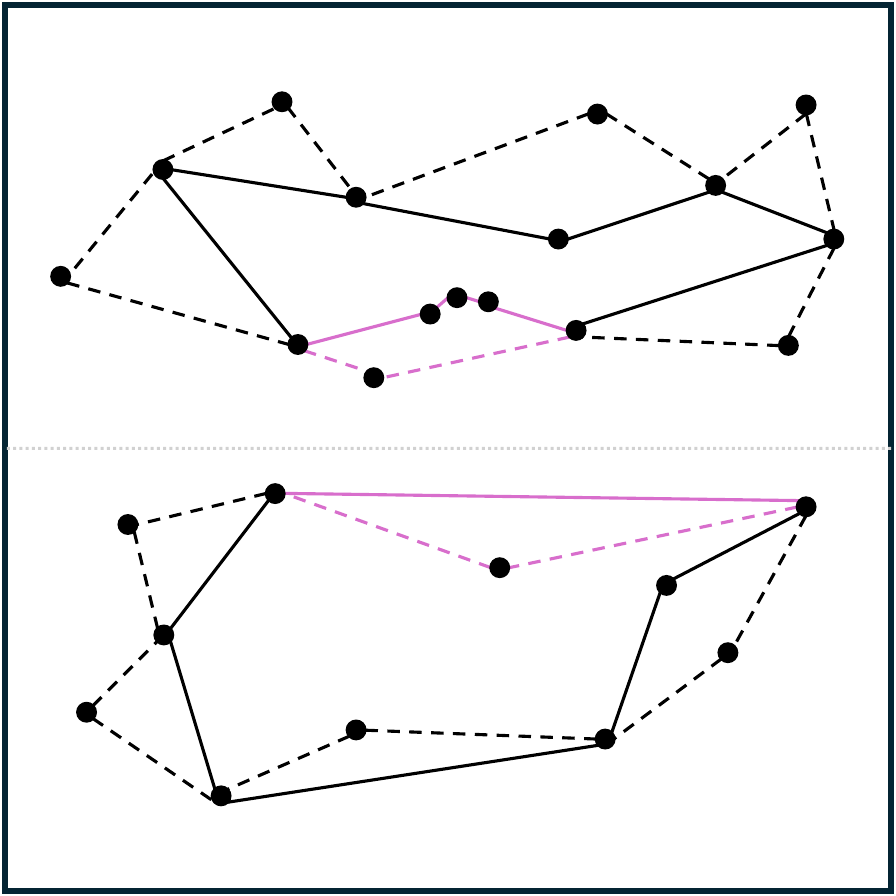}
    \caption{Optimal TSPD solutions, $\Rc_1$ and $\Rc_2$, for $\Xc_1$ and $\Xc_2$, with two chosen rings (shown in purple).}
    \label{fig_tspd2}
\end{subfigure}
\begin{subfigure}[t]{0.48\textwidth}
    \centering
    \includegraphics[width=0.9\linewidth]{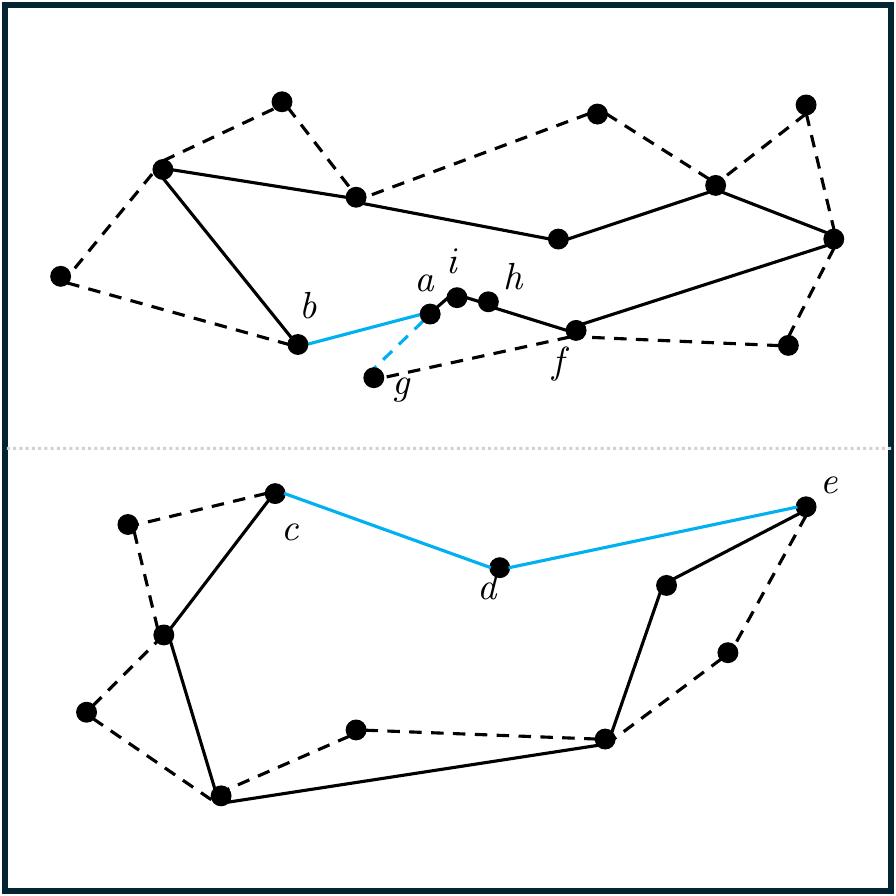}
    \caption{Exposing straight rings (shown in solid blue) from the chosen rings.}
    \label{fig_tspd3}
\end{subfigure}\hfill
\begin{subfigure}[t]{0.48\textwidth}
    \centering
    \includegraphics[width=0.9\linewidth]{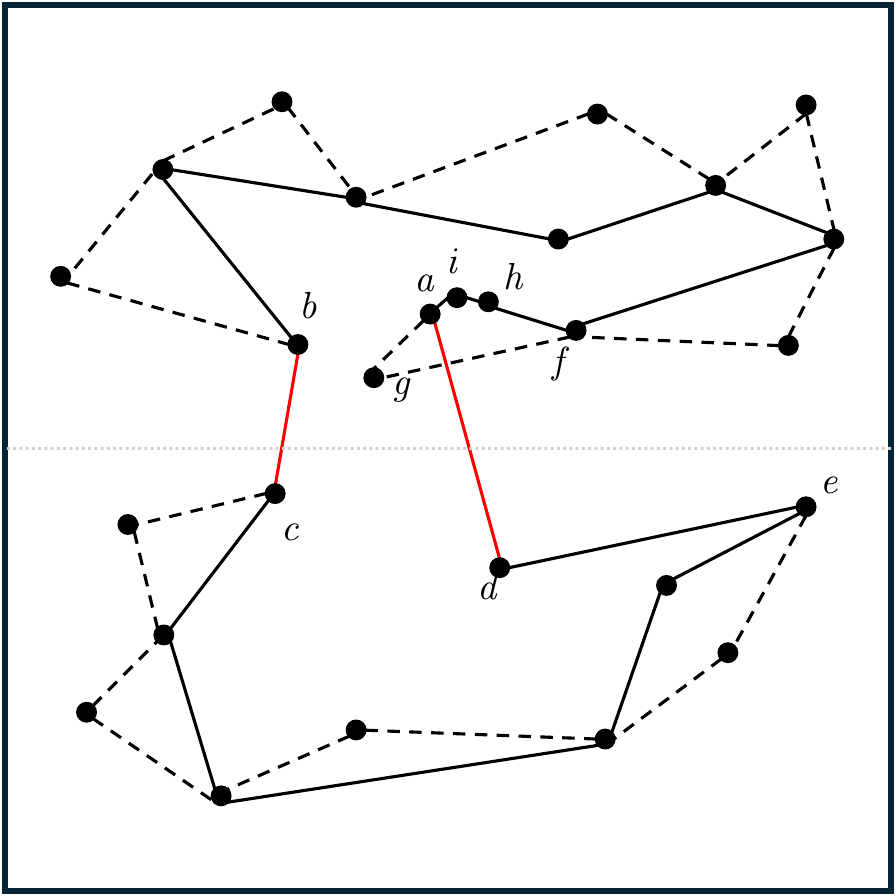}
    \caption{Removing the exposed straight rings and connecting the two local solutions with newly added straight rings (shown in red).}
    \label{fig_tspd4}
\end{subfigure}

\caption{a procedure for combining two local TSPD tours.}
\label{fig_tspd_subadditivity}
\end{figure}

\begin{theorem}\label{thm_tspd_limit}
If $X_1, \ldots, X_n$ are independent and uniformly distributed in $[0, 1]^2$, then there exists a constant $\beta_\TSPD$ such that 
\begin{equation}
	\lim_{n \to \infty} \frac{\TSPD(\{X_1, \ldots, X_n\})}{\sqrt{n}} = \beta_\TSPD,
\end{equation}
almost surely.
\end{theorem}

\begin{proof}

We show that (A0)--(A2), (A5$''$), and (B1) hold for $\TSPD(\cdot)$.

\begin{description}[]
\item[\rm (A0)] This clearly holds.

\item[\rm (A1) \it Positive Homogeneity.] Since norms are homogeneous, $t(\lambda \Tc_r) = \lambda t(\Tc_r)$ and $d(\lambda \Dc_r) = \lambda d(\Dc_r)$.
Then, we can easily show that $c(\lambda \Xc_r) = \lambda c(\Xc_r)$; consequently, $\TSPD(\lambda \Xc) = \lambda \TSPD(\Xc)$.

\item[\rm (A2) \it Translational Invariance.] This clearly holds.

\item[\rm (A5$''$) \it Rectangle Form of Geometric Subadditivity.] 
Let $\Qc$ be a rectangle in $\Rb^2$ that is split into two rectangles with disjoint interiors, $\Qc_1$ and $\Qc_2$, so that $\Qc = \Qc_1 \cup \Qc_2$ and $\Qc_1 \cap \Qc_2 = \emptyset$. 
For any finite point set $\Xc \subset \Rb^2$, define $\Xc_i := \Xc \cap \Qc_i$.
Assign points on the common boundary to either side arbitrarily so that $\Xc = \Xc_1 \cup \Xc_2$ and $\Xc_1 \cap \Xc_2 = \emptyset$. 
Let $\Rc$, $\Rc_1$, and $\Rc_2$ be an optimal ring family for $\Xc$, $\Xc_1$, and $\Xc_2$, respectively.

We first prove the claim in general.
Since all norms on $\Rb^2$ are equivalent, there exist constants $C_T,C_D>0$ such that
\[
	\|p\|_T\leq C_T\|p\|_2,
	\qquad
	\|p\|_D\leq C_D\|p\|_2
	\qquad
	\text{for all }p\in\Rb^2.
\]
Consider an optimal TSPD solution $\Rc_i$ for $\Xc_i$, $i=1,2$.
If one of the two sets has at most two points, the degenerate-case convention bounds its contribution by a constant multiple of $\diam\Qc$, so the desired inequality clearly holds.
We, therefore, assume that both sets contain at least three points.

We first expose one straight ring in each local solution.
Choose any ring in $\Rc_i$.
We claim that this ring can be modified, with an increase of at most $2C_T\diam\Qc$, so that $\Rc_i$ contains a straight ring.
If the chosen ring is already a straight ring, no modification is required.
If the chosen ring is triangular, with combined nodes $s,t$ and drone node $w$, replace it by the two straight rings $sw$ and $wt$.
The increase in cost is bounded by
\[
	\begin{aligned}
	\|s-w\|_T+\|w-t\|_T -\max\{\|s-t\|_T,\|s-w\|_D&+\|w-t\|_D\} \\
	&\leq \|s-w\|_T+\|w-t\|_T \leq 2C_T\diam\Qc .
	\end{aligned}
\]
Finally, suppose that the chosen ring contains truck nodes.
Let $s$ and $t$ be the combined endpoints of the ring, and let $u$ be the truck node adjacent to $s$ along the truck path.
Treat $u$ as a combined node in the modified solution, and replace the original ring by the straight ring $su$ and the residual ring from $u$ to the other combined endpoint.
If the original ring contains a drone node $b$, change the corresponding drone path so that it starts from $u$ rather than from $s$.
Let $P_{u,t}$ be the remaining truck path from $u$ to $t$.
The increase in cost is bounded by
\[
	\begin{aligned}
	\|s-u\|_T+ \max\{t(P_{u,t}), &\|u-b\|_D+\|b-t\|_D\}\\
	&-\max\{\|s-u\|_T+t(P_{u,t}),\|s-b\|_D+\|b-t\|_D\}\\
	&\leq \|s-u\|_T + \max\{\|u-b\|_D-\|s-b\|_D,0\}\\
	&\leq \|s-u\|_T+\|u-s\|_D\\ 
	&\leq 2C_T\diam\Qc .
	\end{aligned}
\]
The last inequality holds due to the drone-metric dominance.
Hence, in all cases, after increasing the cost by at most $2C_T\diam\Qc$, each local solution contains an exposed straight ring.
Let these exposed straight rings have endpoints $\{a,b\}$ in $\Rc_1$ and $\{c,d\}$ in $\Rc_2$.

Remove these two exposed straight rings, and add the two straight rings $ad$ and $bc$.
These two rings concatenate the two paths into a single feasible TSPD solution for $\Xc$.
The cost of the two new connectors is bounded by $2C_T\diam\Qc$.
Therefore, accounting for the two exposures and the concatenation,
\begin{equation}
\label{eq:tspd_geometric_subadditivity}
\TSPD(\Xc)\leq\TSPD(\Xc_1)+\TSPD(\Xc_2)+6C_T\diam\Qc.
\end{equation}
This proves (A5$''$).

\item[\rm (B1) \it Smoothness.]
Let $\Fc,\Gc\subset [0,1]^2$ be finite sets.
Since $\Fc\cup\Gc=\Fc\cup(\Gc\setminus\Fc)$ and $|\Gc\setminus\Fc|\leq|\Gc|$, we may replace $\Gc$ by $\Gc\setminus\Fc$.
Thus, without loss of generality, assume that $\Fc\cap\Gc=\emptyset$, and put $k=|\Gc|$.
The case $k=0$ is trivial, so assume $k\geq1$.

Let $D_0:=\diam([0,1]^2)$.
We use the deterministic bound that any $k$ points in the unit square can be connected by a straight-ring-only TSPD tour of cost at most $C_0\sqrt{k}$, 
where $C_0$ depends only on the truck norm.
Indeed, \citet{few1955bound} showed that any fixed set of $k$ points in the unit square admits a Euclidean path of length $O(\sqrt{k})$.
Closing this path and using $\|p\|_T\leq C_T\|p\|_2$ gives a truck tour of length at most $C_0\sqrt{k}$, which is exactly a straight-ring-only TSPD tour.

We show that $\TSPD(\Fc\cup\Gc) - \TSPD(\Fc)$ is bounded from above by a factor of $O(\sqrt{k})$.
First, take an optimal TSPD solution for $\Fc$.
By the degenerate-case convention, the cases $|\Fc|\leq2$ or $k\leq2$ are absorbed into the constant term.
Thus assume $|\Fc|\geq3$ and $k\geq3$.
By the straight-ring exposure argument used in proving \eqref{eq:tspd_geometric_subadditivity}, after increasing its cost by at most $2C_TD_0$, the solution for $\Fc$ contains an exposed straight ring, say $ab$.

Next, connect the points in $\Gc$ by a straight-ring-only tour of cost at most $C_0\sqrt{k}$.
Choose one straight ring $cd$ in this tour.
Remove the straight rings $ab$ and $cd$, and add the straight rings $ad$ and $bc$.
This produces a feasible TSPD solution for $\Fc\cup\Gc$, and the two added connectors have total cost at most $2C_TD_0$.
Therefore,
\begin{equation}
\label{eq:tspd_smoothness_forward}
\begin{aligned}
	\TSPD(\Fc\cup\Gc)
	&\leq
	\TSPD(\Fc)+2C_TD_0+C_0\sqrt{k}
	\\
	&\quad
	+\|a-d\|_T+\|b-c\|_T-\|a-b\|_T-\|c-d\|_T \\
	&\leq
	\TSPD(\Fc)+C_0\sqrt{k}+4C_TD_0 \\
	&\leq 
	\TSPD(\Fc)
	+\bigl(C_0+4C_TD_0\bigr)\sqrt{k},
\end{aligned}
\end{equation}
since $k \geq 3$.

Now we show that $\TSPD(\Fc\cup\Gc) - \TSPD(\Fc)$ is bounded from below by a factor of $O(\sqrt{k})$.
Let $\Rc$ be an optimal TSPD solution for $\Fc\cup\Gc$.
Let $\Ac\subseteq\Fc$ be the set of points that are served by the drone in $\Rc$ and have at least one combined endpoint in $\Gc$.
Each removed combined node is incident to two adjacent rings, and each ring contains at most one drone node.
Thus, $|\Ac|\leq 2k$.
Remove the points of $\Gc\cup\Ac$ from $\Rc$.
If a removed point is a truck node, we shortcut the truck path; this does not increase the truck cost by the triangle inequality.
If a removed point is a drone node, we delete the corresponding drone edges; this does not increase the ring cost.

It remains to repair the rings incident to the removed combined nodes.
For each such ring, if the drone path has the removed combined node as an endpoint, delete that drone path.
This does not increase the cost of the ring.
Also, replace the remaining truck path, adjacent to the removed combined node, with straight rings between consecutive truck nodes, and reassign these truck nodes as combined nodes.
This replacement does not increase the cost, since the cost of each such straight ring is equal to the corresponding truck-edge cost.

After these local repairs, the remaining tour structure decomposes into maximal connected residual chains, where connectivity is defined through remaining combined nodes.
If no residual chain remains, then $\Fc\setminus\Ac$ is covered by the degenerate-case convention, and the desired inequality is immediate.
For some integer $q\geq 1$, denote these chains in the cyclic order inherited from $\Rc$ by
\[
	\Cc_1,\Cc_2,\ldots,\Cc_q .
\]
In other words, $\Cc_j$ and $\Cc_{j+1}$ are consecutive if no other residual chain lies between them in the cyclic order induced by $\Rc$.
For each $\Cc_j$, let $a_j$ be the terminal combined node adjacent to $\Cc_{j-1}$ in this order, and let $b_j$ be the terminal combined node adjacent to $\Cc_{j+1}$.
To form a feasible TSPD solution on $\Fc\setminus\Ac$ from these chains, add straight rings $b_j a_{j+1}$, with the convention that $a_{q+1}=a_1$.
Suppose a removed combined node $g\in \Gc$ separated $\Cc_j$ and $\Cc_{j+1}$.
Then, the new straight ring $b_j a_{j+1}$ satisfies
\[
	\|b_j-a_{j+1}\|_T
	\leq
	\|b_j-g\|_T+\|g-a_{j+1}\|_T
\]
by the triangle inequality.
Thus, these straight rings cost no more than the truck paths they replace in $\Rc$.
Together with the preceding deletions and local repairs, this gives a feasible TSPD solution on $\Fc\setminus\Ac$ whose total cost is no larger than the cost of $\Rc$.
Therefore,
\begin{equation}
\label{eq:tspd_smoothness_deletion}
	\TSPD(\Fc\setminus\Ac)
	\leq
	\TSPD(\Fc\cup\Gc).
\end{equation}
Viewing $\Fc$ as $(\Fc\setminus\Ac)\cup\Ac$, \eqref{eq:tspd_smoothness_forward} gives
\[
	\TSPD(\Fc)
	\leq
	\TSPD(\Fc\setminus\Ac)
	+
	\bigl(C_0+4C_TD_0\bigr)\sqrt{|\Ac|}.
\]
Then, by \eqref{eq:tspd_smoothness_deletion} and $|\Ac|\leq 2k$, it follows that
\[
	\TSPD(\Fc)
	\leq
	\TSPD(\Fc\cup\Gc)
	+
	\sqrt{2}\bigl(C_0+4C_TD_0\bigr)\sqrt{k}.
\]
This proves (B1).

\end{description}

With (A0)--(A2), (A5$''$), and (B1) holding, we have a proof of the theorem.
\end{proof}

The convergence also holds in expectation.
Indeed, the deterministic bound used for (B1) gives, for some constant $B > 0$,
\[
	0\leq \frac{\TSPD(\{X_1,\ldots,X_n\})}{\sqrt n}\leq B.
\]
Therefore, bounded convergence yields
\begin{equation}
	\lim_{n\to\infty}
	\frac{\Eb[\TSPD(\{X_1,\ldots,X_n\})]}{\sqrt n}
	=\beta_\TSPD.
\end{equation}

\subsection{The Speed-Scaled Euclidean Model}
\label{sec_special_case}

If both truck and drone costs are based on the Euclidean norm, $\|\cdot\|_2$, 
and the drone operates at relative speed $\alpha \geq 1$, then
\[
\|x_i - x_j\|_T = \|x_i - x_j\|_2,
\qquad
\|x_i - x_j\|_D = \frac{1}{\alpha} \|x_i - x_j\|_2.
\]
This model is commonly used in the literature \citep{agatz2018optimization,bogyrbayeva2023deep,mahmoudinazlou2024hybrid}.

\subsection{The Rectilinear-Euclidean Mixed Model} 
\label{sec_mixed_model}

Several TSPD studies \citep{poikonen2019branch,de2020variable} have considered the Rectilinear-Euclidean mixed norm case.
That is, the truck's distance is measured by $\ell_1$-norm, while the drone's distance is measured by the speed-scaled $\ell_2$-norm.
In particular,
\[
\|x_i - x_j\|_T = \|x_i - x_j\|_1,
\qquad
\|x_i - x_j\|_D = \frac{1}{\alpha} \|x_i - x_j\|_2,
\]
for the drone speed parameter $\alpha \geq 1$.

For the result of this paper, we consider this speed-scaled Euclidean model to derive bounds, while we make some comments on the Rectilinear-Euclidean mixed model in Section~\ref{sec_mixed}.

\section{Upper Bounds}

\label{sec_upper_bounds}

We consider the speed-scaled Euclidean distance model in Section~\ref{sec_special_case}.
To derive an explicit upper bound for $\beta_{\TSPD}$, we adapt the classical strip-partition method of \citet{beardwood1959shortest} to the TSPD setting. 
The key idea is to construct a feasible tour by tiling the plane with narrow strips and organizing customers into \emph{rings} of simple structure.

In the formation of a feasible TSPD route, we consider three primitive types: straight, triangle, and quartet rings. 
A \emph{straight ring} consists of two combined nodes connected directly by the truck while the drone is on board; this is called a 2-point ring.
A \emph{triangle ring} consists of two combined nodes with one drone node in between; this is a 3-point ring.
A \emph{quartet ring} consists of two combined nodes, one drone node, and one truck node; this is a 4-point ring.
Similarly, a \emph{quintet ring} consists of two combined nodes, one drone node, and two truck nodes; a 5-point ring.
We will first form a feasible routing solution within each strip, only using homogeneous ring types; for example, a solution with straight rings only, triangle rings only, or quartet rings only.

Consider the set of arbitrary consecutive $k$ points, $X_0, X_1, \ldots X_{k-1}$. 
We introduce unscaled random variables:
\[
Z_1, Z_2,\ldots, Z_{k-1} \ \stackrel{\text{i.i.d.}}{\sim} \mathrm{Exponential}(1),
\qquad
U_0, U_1, \ldots, U_{k-1} \ \stackrel{\text{i.i.d.}}{\sim} \mathrm{Uniform}[0,1],
\]
For each strip with height $h/\sqrt{n}$, 
we denote the unscaled cost of a $k$-point ring by $C_k(h, \alpha)$ for given $h$ and the drone speed factor $\alpha$.
For example, the triangle ring is a 3-point ring with $k=3$.
For brevity, we will also use $C_k$ to denote $C_k(h, \alpha)$ whenever the meaning is clear.

\begin{figure}
\centering
\begin{subfigure}[t]{\textwidth}
    \centering
    \includegraphics[width=0.6\linewidth]{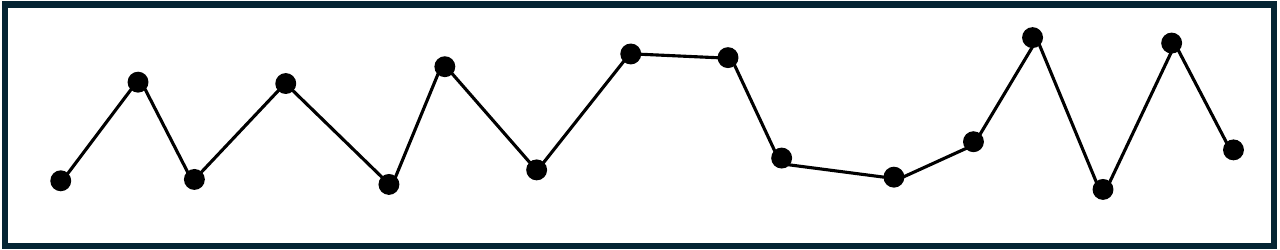}
    \caption{Straight rings only}
    \label{fig_straight}
\end{subfigure}

\vspace{1em} %

\begin{subfigure}[t]{\textwidth}
    \centering
    \includegraphics[width=0.6\linewidth]{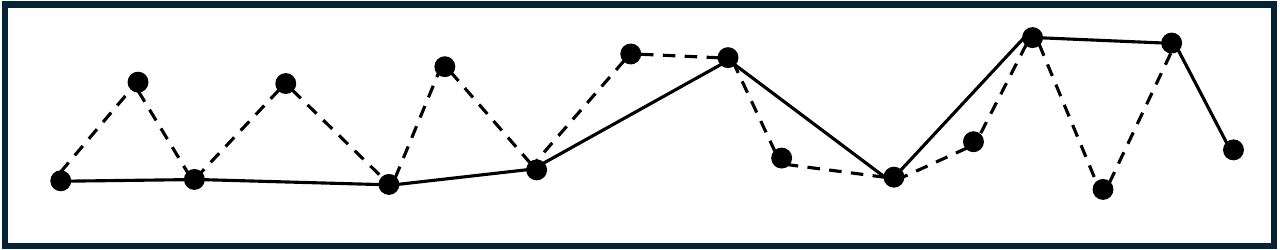}
    \caption{Triangle rings only, with one ``leftover'' in the far right}
    \label{fig_triangle}
\end{subfigure}

\vspace{1em} %

\begin{subfigure}[t]{\textwidth}
    \centering
    \includegraphics[width=0.6\linewidth]{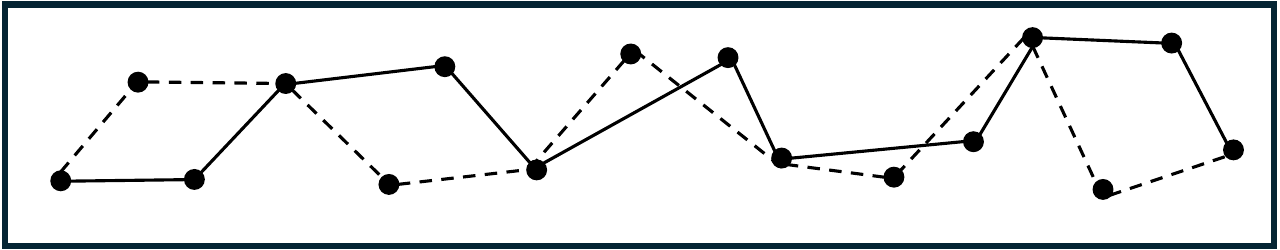}
    \caption{Quartet rings only}
    \label{fig_quartet}
\end{subfigure}
\caption{Three types of rings.}
\label{fig_homogeneous_rings}
\end{figure}

\subsection{Straight Rings Only}\label{sec_straight}
Consider a straight ring, where the truck and the drone travel together at the speed of the truck.
As seen in the TSP case, the unscaled distance between $X_0$ and $X_1$ can be written as:
\[
	L_{01} = \sqrt{ Z_1^2 + h^4 (U_0 - U_1)^2 }.
\]
This becomes the cost of the straight ring directly:
\[
	C_2 = L_{01}.
\]
If we construct a TSPD path using straight rings only, as shown in Figure~\ref{fig_straight}, we can conclude that the limit is bounded by
\[
\beta_\TSPD \leq 
\frac{1}{h} \Eb[C_2] = \frac{1}{h} \int_0^\infty \int_0^1 \int_0^1 e^{-z_1} \sqrt{z_1^2 + h^4(u_0-u_1)^2} \du_0 \du_1 \dz_1.
\]
Since $\Eb[Z_1^2]=2$ and $\Eb[(U_0-U_1)^2]=1/6$, Jensen's inequality gives
\[
	\Eb[C_2]\leq \sqrt{2+\frac{h^4}{6}}.
\]
Optimizing over $h$ immediately yields the following analytical bound:
\begin{equation}\label{eq_straight_analytical_bound}
	\beta_\TSPD
	\leq \inf_{h>0}\frac{1}{h}\sqrt{2+\frac{h^4}{6}}
	= \sqrt{\frac{2}{\sqrt{3}}}
	\approx 1.0746,
\end{equation}
where the infimum is attained at $h=12^{1/4}$.

\subsection{Triangle Rings Only} \label{sec_triangle}
We consider $\{ X_0, X_1, X_2 \}$ and the unscaled distances:
\begin{align*}
	L_{01} &= \sqrt{ Z_1^2 + h^4 (U_0 - U_1)^2 }, \\
	L_{12} &= \sqrt{ Z_2^2 + h^4 (U_1 - U_2)^2 }, \\
	L_{02} &= \sqrt{ (Z_1 + Z_2)^2 + h^4 (U_0 - U_2)^2 }. 
\end{align*}
Forming triangle rings, we only consider the mid-point, $X_1$, as the drone node. 
This is to ensure the connectivity and the independence between consecutive triangle rings. 
The cost of the triangle ring is
\[
	C_3 = \max\Big\{ L_{02}, \frac{1}{\alpha} (L_{01} + L_{12}) \Big\}
\]
Note that each triangle ring consumes 2 additional points.
Depending on the number of points within each strip, it is possible that we have a ``leftover'' node that cannot be part of a triangle ring, as shown in Figure~\ref{fig_triangle}.
We need to connect leftover nodes by a straight ring.
However, such parts have a total length of $o(\sqrt{n})$, hence they do not impact $\beta_\TSPD$.
Using triangle-ring-only TSPD paths, we conclude that the limit is bounded by 
\begin{multline}
\beta_\TSPD \leq 
	\frac{1}{2h} \Eb[C_3]\\
	 = 
	\frac{1}{2h} \int_0^\infty \int_0^\infty \int_{[0,1]^3}
e^{-z_1-z_2} 
\max\bigg\{
\sqrt{(z_1+z_2)^2 + h^4(u_0-u_2)^2}, \\
\frac{1}{\alpha}\bigg( \sqrt{z_1^2 + h^4(u_0-u_1)^2} + \sqrt{z_2^2 + h^4(u_1-u_2)^2}\bigg)
\bigg\}
\du \dz_1 \dz_2,
\end{multline}
where we used $u = (u_0, u_1, u_2)$ as a three-dimensional vector.

Since $\max\{a,b\}\leq\sqrt{a^2+b^2}$ and $(a+b)^2\leq 2(a^2+b^2)$ for any $a,b\geq 0$, Jensen's inequality gives
\[
	\Eb[C_3]
	\leq
	\sqrt{\Eb[L_{02}^2]+\frac{2}{\alpha^2}\Eb[L_{01}^2+L_{12}^2]}
	=
	\sqrt{6+\frac{h^4}{6}+\frac{4}{\alpha^2}\left(2+\frac{h^4}{6}\right)},
\]
where $\Eb[L_{02}^2]=6+h^4/6$ and $\Eb[L_{01}^2]=\Eb[L_{12}^2]=2+h^4/6$.
Optimizing over $h$ immediately yields the following analytical bound:
\begin{equation}\label{eq_triangle_analytical_bound}
	\beta_\TSPD
	\leq
	\frac{1}{\sqrt{2}}
	\left[
		\left(1+\frac{4}{3\alpha^2}\right)
		\left(1+\frac{4}{\alpha^2}\right)
	\right]^{1/4},
\end{equation}
where the infimum is attained at $h=\left[12(3\alpha^2+4)/(\alpha^2+4)\right]^{1/4}$.

\subsection{Quartet Rings Only}\label{sec_quartet}

To extend our discussion to quartet rings and other more complicated cases, we define
\begin{align}
	W_i &= \sum_{j=0}^i Z_j, \quad i = 0, 1, 2, \ldots \label{eq_general00} \\
	W_0 &= Z_0 = 0, \\
	L_{ij} &= \sqrt{
		(W_j - W_i)^2 + h^4(U_i - U_j)^2
	}, \quad i, j = 0, 1, 2, \ldots \\
	L_{ij} &= L_{ji}. \label{eq_general99}
\end{align}
We consider a quartet ring consists of $\{ X_0, X_1, X_2, X_3\}$ and the unscaled distances of $L_{01}$, $L_{12}$, $L_{23}$, $L_{02}$, and $L_{13}$.
While $X_0$ and $X_3$ are fixed as endpoints, we have options to choose either $X_1$ or $X_2$ as a drone node.
In the decision making, we, of course, want to make $C_4$ as small as possible for a tighter upper bound.
That is,
\begin{equation*}
	C_4 =
	\min\bigg\{
	\max\Big\{ L_{01} + L_{13}, \frac{1}{\alpha} (L_{02} + L_{23}) \Big\},
	\max\Big\{ L_{02} + L_{23}, \frac{1}{\alpha} (L_{01} + L_{13}) \Big\}
	\bigg\}.
\end{equation*}

The calculation of $C_4$ can be slightly simplified.

\begin{lemma} \label{lem_c4}
We can show that
\begin{equation}
	C_4 = \max \Big\{ T_{\min}, \frac{1}{\alpha} T_{\max} \Big\}
\end{equation}
where $T_{\min} = \min\{L_{01}+L_{13}, L_{02} + L_{23}\}$ and $T_{\max} = \max\{L_{01}+L_{13}, L_{02} + L_{23}\}$.
\end{lemma}

\begin{proof}
Among $\{X_0, X_1, X_2, X_3\}$, $X_0$ and $X_3$ are fixed as the combined nodes.
There are two possible cases: $X_1$ is the drone node, or $X_2$.
Comparing these two cases, we take the minimum.
We can show that
\begin{align*}
	C_4 
	&= \min \Bigg\{
		\max \Big\{ L_{01} + L_{13}, \frac{1}{\alpha} (L_{02} + L_{23}) \Big\}, \ 
		\max \Big\{ L_{02} + L_{23}, \frac{1}{\alpha} (L_{01} + L_{13}) \Big\}
	\Bigg\} \\
	&= \max \Big\{ \min\{L_{01}+L_{13}, L_{02} + L_{23}\}, \ \frac{1}{\alpha}  \max\{L_{01}+L_{13}, L_{02} + L_{23} \} \Big\}.
\end{align*}
This completes the proof.
\end{proof}

Since we can determine $T_{\min}$ and $T_{\max}$ from the same comparison operation, Lemma~\ref{lem_c4} will help reduce computational effort slightly.

After forming a series of quartet rings within a strip, we may have one or two leftover points at the end, which we connect by straight rings.
Again, such straight rings have a total length of $o(\sqrt{n})$ and therefore do not affect $\beta_\TSPD$.
Since each quartet ring, a 4-point ring, consumes $3$ additional points, by using only quartet rings we can bound the limit as
\[
\beta_\TSPD \leq \frac{1}{3h} \Eb[C_4]
	= \frac{1}{3h} \int_{0}^{\infty} \int_{0}^{\infty} \int_{0}^{\infty} \int_{[0,1]^4} e^{-(z_1+z_2+z_3)}
		\max\{T_{\min}, \tfrac{1}{\alpha}T_{\max} \} 
	\du \dz_1 \dz_2 \dz_3,
\]
where we used $u = (u_0, u_1, u_2, u_3)$ as a vector.

Let $S_1=L_{01}+L_{13}$ and $S_2=L_{02}+L_{23}$. Since $\max\{a,b\}^2\leq a^2+b^2$ for any $a,b\geq 0$,
\[
	C_4^2
	\leq T_{\min}^2+\frac{1}{\alpha^2}T_{\max}^2
	\leq \left(\frac{1}{2}+\frac{1}{\alpha^2}\right)(S_1^2+S_2^2).
\]
Thus, Jensen's inequality gives
\[
	\Eb[C_4]
	\leq
	\sqrt{\left(\frac{1}{2}+\frac{1}{\alpha^2}\right)\Eb[S_1^2+S_2^2]}.
\]
Applying $(a+b)^2\leq 2(a^2+b^2)$ to $S_1$ and $S_2$ gives
\[
	\Eb[S_1^2+S_2^2]\leq 32+\frac{4h^4}{3}.
\]
Therefore, optimizing over $h$ immediately yields the following analytical bound:
\begin{equation}\label{eq_quartet_analytical_bound}
	\beta_\TSPD
	\leq
	\frac{1}{3}
	\sqrt{\frac{1}{2}+\frac{1}{\alpha^2}}
	\sqrt{2\sqrt{\frac{128}{3}}},
\end{equation}
where the infimum is attained at $h=24^{1/4}$.

\subsection{General Cases of Homogeneous Ring Types}
So far, we have considered 2-point rings (straight), 3-point rings (triangle), and 4-point rings (quartet).
We can continue to consider 5-point rings, 6-point rings, etc.
In general, with $C_k$ for an integer $k \geq 2$, we can write
\begin{equation}
	\beta_{\TSPD}(\alpha) \le \inf_{h>0} \frac{1}{(k-1)h} \Eb[C_k(h,\alpha)]
\end{equation}
for any given $\alpha \geq 1$.
We can also write
\[
	\beta_{\TSPD}(\alpha) \le \min_{k \geq 2} \inf_{h>0} \frac{1}{(k-1)h} \Eb[C_k(h,\alpha)].
\]	
However, considering many $k \geq 2$ values will be impractical for evaluating the bound, if not impossible.
We limit our choice of $k$ up to 2, 3, and 4, for homogeneous-ring formations as in Sections \ref{sec_straight} to \ref{sec_quartet}.

\subsection{Alternating among $k$-point Ring Patterns}

Instead of sticking to a single type of ring, we can also alternate among different types of rings.
When $k=2$ and $k=3$, the only possibilities are a straight ring and a triangle ring, respectively, that are determined uniquely.
We consider when $k \geq 4$.

Before we discuss alternating patterns, we make an important observation.
It is known that there exists an optimal TSPD solution without two consecutive straight rings \citep{lee2025iterative}. 
We strengthen this result.

\begin{figure}
\centering
\includegraphics[width=\textwidth]{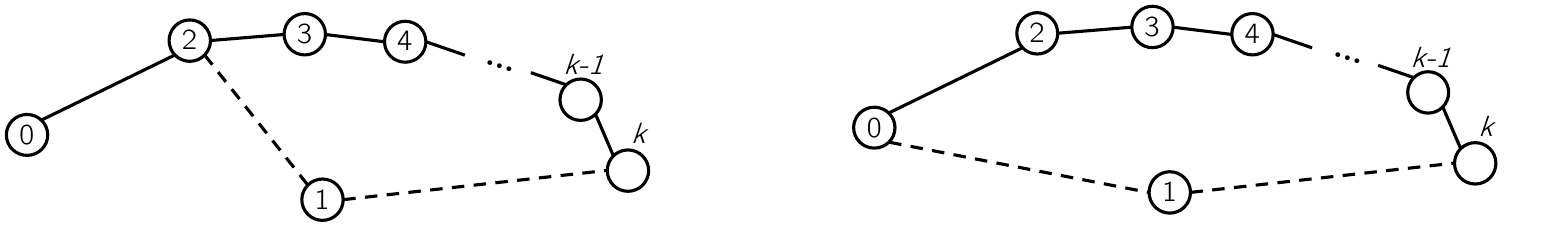}
\caption{We can merge a straight ring to a neighboring ring to create a solution no worse than the current one.}
\label{fig_no_straight_ring}
\end{figure}

\begin{theorem} \label{thm_no_straight_ring}
Suppose the drone metric is no larger than the truck metric.
For any given TSPD, there exists an optimal solution without any straight rings.
\end{theorem}

\begin{proof}
Suppose there exists a straight ring as in the left figure of Figure~\ref{fig_no_straight_ring}, which is next to a $k$-point ring, wherein node 1 is the drone node.
Let $\ell^T_{ij} = \|x_i-x_j\|_T$ and $\ell^D_{ij} = \|x_i-x_j\|_D$.
We show that 
\begin{multline*}
	\ell^T_{02} + \max\Big\{ \ell^T_{23} + \cdots + \ell^T_{k-1, k}, \ell^D_{21} + \ell^D_{1k} \Big\}
	   \geq 
	\max\Big\{ \ell^T_{02} + \ell^T_{23} + \cdots + \ell^T_{k-1, k}, \ell^D_{01} + \ell^D_{1k} \Big\}
\end{multline*}
where the left-hand side represents the left figure and the right-hand side represents the right figure in Figure~\ref{fig_no_straight_ring}.
From the triangle 012, we have $\ell^D_{01} \leq \ell^D_{02} + \ell^D_{21}$.
By the metric dominance assumption, we obtain
\[
	\ell^D_{01} + \ell^D_{1k}
	\leq
	\ell^D_{02} + \ell^D_{21} + \ell^D_{1k}
	\leq
	\ell^D_{21} + \ell^D_{1k} + \ell^T_{02}.
\]
Hence, we obtain
\begin{align*}
	& \ell^T_{02} + \max\Big\{ \ell^T_{23} + \cdots + \ell^T_{k-1, k}, \ell^D_{21} + \ell^D_{1k} \Big\} \\
	& = \max\Big\{ \ell^T_{02} + \ell^T_{23} + \cdots + \ell^T_{k-1, k}, \ell^T_{02} + \ell^D_{21} + \ell^D_{1k} \Big\} \\
	& \geq 	\max\Big\{ \ell^T_{02} + \ell^T_{23} + \cdots + \ell^T_{k-1, k}, \ell^D_{01} + \ell^D_{1k} \Big\},
\end{align*}
which shows that you can merge a straight ring to a neighboring ring without increasing the total cost.
\end{proof}

This theorem will help us simplify the calculations.

\subsubsection{Four-point Ring Patterns} 

We assume that we consider a block of 4 points each time, and determine the best choice of ring composition.
There are three cases:
(i) Triangle-Straight,
(ii) Straight-Triangle, and
(iii) Quartet.
Theorem~\ref{thm_no_straight_ring}, however, tells that Triangle-Straight and Straight-Triangle patterns are no better than the Quartet patterns.
That is, by merging the straight ring with the neighboring triangle ring, we obtain a quartet ring without increasing the objective function value.
Therefore, for four-point ring patterns, we only need to consider quartet rings as done in Section~\ref{sec_quartet}.

\subsubsection{Five-point Ring Patterns}\label{sec_five_point_patterns}

\begin{figure}
\begin{subfigure}[t]{0.48\textwidth}
    \centering
    \includegraphics[width=0.8\linewidth]{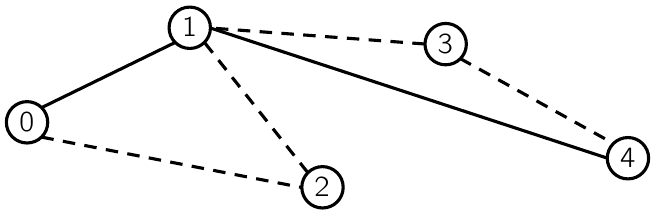}
    \caption{$\text{Triangle}(0, 1; 2) + \text{Triangle}(1, 4; 3)$}
\end{subfigure}\hfill
\begin{subfigure}[t]{0.48\textwidth}
    \centering
    \includegraphics[width=0.8\linewidth]{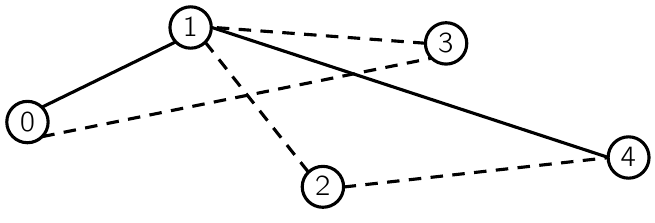}
    \caption{$\text{Triangle}(0, 1; 3) + \text{Triangle}(1, 4; 2)$}
\end{subfigure}

\vspace{1em} %

\begin{subfigure}[t]{0.48\textwidth}
    \centering
    \includegraphics[width=0.8\linewidth]{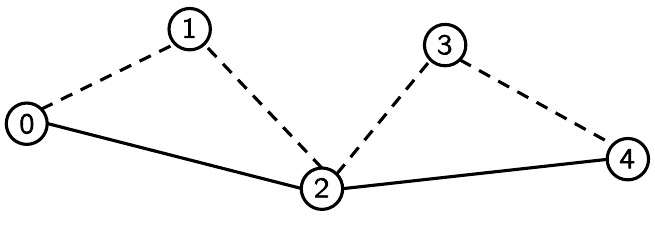}
    \caption{$\text{Triangle}(0, 2; 1) + \text{Triangle}(2, 4; 3)$}
\end{subfigure}\hfill
\begin{subfigure}[t]{0.48\textwidth}
    \centering
    \includegraphics[width=0.8\linewidth]{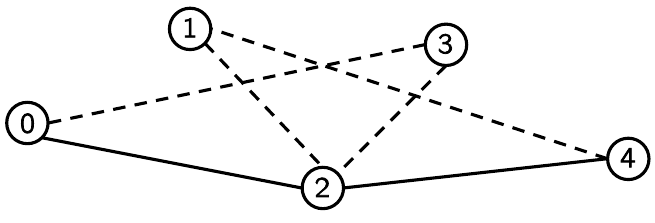}
    \caption{$\text{Triangle}(0, 2; 3) + \text{Triangle}(2, 4; 1)$}
\end{subfigure}

\vspace{1em} %

\begin{subfigure}[t]{0.48\textwidth}
    \centering
    \includegraphics[width=0.8\linewidth]{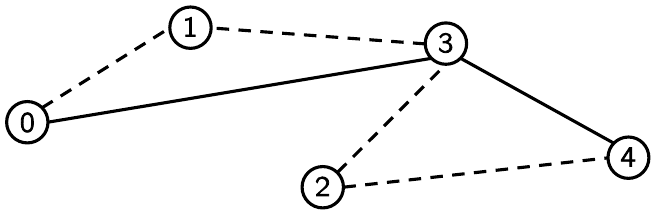}
    \caption{$\text{Triangle}(0, 3; 1) + \text{Triangle}(3, 4; 2)$}
\end{subfigure}\hfill
\begin{subfigure}[t]{0.48\textwidth}
    \centering
    \includegraphics[width=0.8\linewidth]{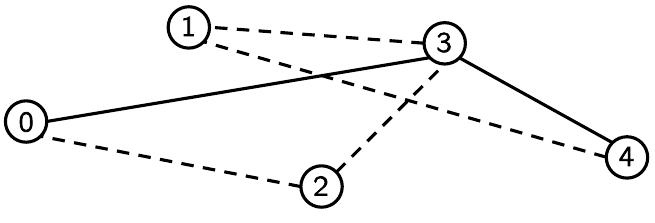}
    \caption{$\text{Triangle}(0, 3; 2) + \text{Triangle}(3, 4; 1)$}
\end{subfigure}

\caption{
Triangle-Triangle patterns with five points. 
Note that only the ordering of the horizontal coordinates is important in this figure;
therefore, the shape shown in the figure may not represent the actual distance.
}
\label{fig_tt}
\end{figure}

\begin{figure}
\begin{subfigure}[t]{0.48\textwidth}
    \centering
    \includegraphics[width=0.8\linewidth]{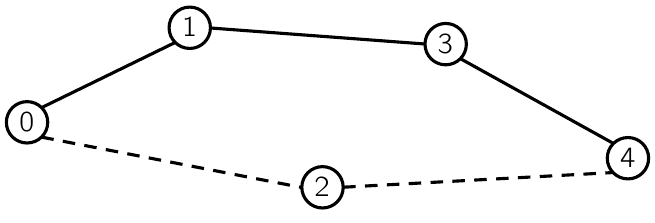}
    \caption{$\text{Quintet}(0, 1, 3, 4; 2)$}
\end{subfigure}\hfill
\begin{subfigure}[t]{0.48\textwidth}
    \centering
    \includegraphics[width=0.8\linewidth]{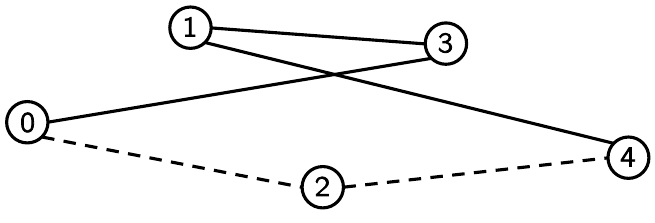}
    \caption{$\text{Quintet}(0, 3, 1, 4; 2)$}
\end{subfigure}

\vspace{1em} %

\begin{subfigure}[t]{0.48\textwidth}
    \centering
    \includegraphics[width=0.8\linewidth]{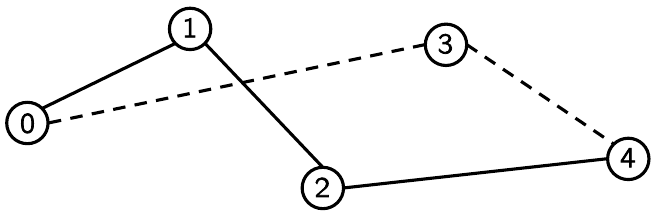}
    \caption{$\text{Quintet}(0, 1, 2, 4; 3)$}
\end{subfigure}\hfill
\begin{subfigure}[t]{0.48\textwidth}
    \centering
    \includegraphics[width=0.8\linewidth]{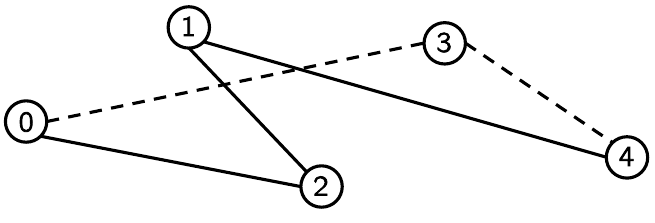}
    \caption{$\text{Quintet}(0, 2, 1, 4; 3)$}
\end{subfigure}

\vspace{1em} %

\begin{subfigure}[t]{0.48\textwidth}
    \centering
    \includegraphics[width=0.8\linewidth]{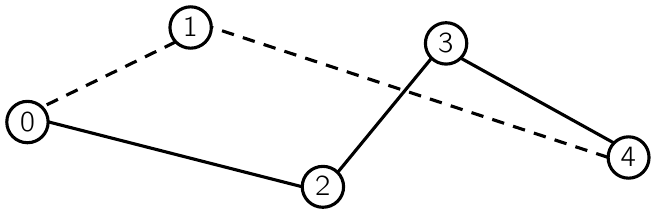}
    \caption{$\text{Quintet}(0, 2, 3, 4; 1)$}
\end{subfigure}\hfill
\begin{subfigure}[t]{0.48\textwidth}
    \centering
    \includegraphics[width=0.8\linewidth]{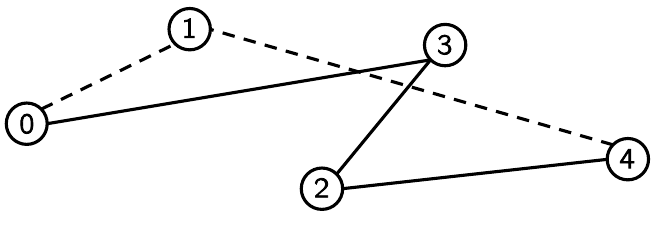}
    \caption{$\text{Quintet}(0, 3, 2, 4; 1)$}
\end{subfigure}

\caption{
Quintet patterns. 
Note that only the ordering of the horizontal coordinates is important in this figure;
therefore, the shape shown in the figure may not represent the actual distance.
}
\label{fig_qq}

\end{figure}

We now consider a block of 5 points each time.
We consider the following five cases:
(i) Triangle-Triangle,
(ii) Quartet-Straight,
(iii) Straight-Quartet,
(iv) Straight-Triangle-Straight, and
(v) Quintets.
Again, by Theorem~\ref{thm_no_straight_ring}, we can merge all straight rings in cases (ii), (iii), and (iv), to obtain a quintet ring.
Therefore, we only consider the Triangle-Triangle patterns, presented in Figure~\ref{fig_tt}, and the Quintet patterns, presented in Figure~\ref{fig_qq}.

There are 12 total possibilities. 
Similar to the previous discussion, we define $L_{01}$, $L_{12}$, $L_{23}$, $L_{34}$, $L_{02}$, $L_{03}$, $L_{13}$, and $L_{14}$ using \eqref{eq_general00}--\eqref{eq_general99}.
Let us define the following shorthands:
\begin{align*}
	\text{Triangle}(i, j; k) = \max\Big\{ L_{ij}, \frac{1}{\alpha} (L_{ik} + L_{kj}) \Big\}
\end{align*}
to represent the triangle ring where the truck travels from node $i$ to node $j$ and the drone visits node $k$, and
\begin{align*}
	\text{Quintet}(i, i_1, i_2, j; k) = \max\Big\{ L_{i i_1} + L_{i_1 i_2} + L_{i_2 j}, \frac{1}{\alpha} (L_{ik} + L_{kj}) \Big\}
\end{align*}
to represent the quintet ring where the truck travels from node $i$ to $i_1$ to $i_2$ to $j$ and the drone visits node $k$.

After considering all 12 possibilities, we take the minimum and multiply it by $\frac{1}{4h}$ to obtain the following upper bound:
\begin{align*}
	\frac{1}{4h} \Eb \Bigg[
	\min\bigg\{ 
	& \text{Triangle}(0, 1; 2) + \text{Triangle}(1, 4; 3),  \
	  \text{Triangle}(0, 1; 3) + \text{Triangle}(1, 4; 2),  \\
	& \text{Triangle}(0, 2; 1) + \text{Triangle}(2, 4; 3),  \
	  \text{Triangle}(0, 2; 3) + \text{Triangle}(2, 4; 1),  \\
	& \text{Triangle}(0, 3; 1) + \text{Triangle}(3, 4; 2),  \
	  \text{Triangle}(0, 3; 2) + \text{Triangle}(3, 4; 1),  \\
	& \text{Quintet}(0, 1, 3, 4; 2), \
	  \text{Quintet}(0, 3, 1, 4; 2), \
	  \text{Quintet}(0, 1, 2, 4; 3), \\ 
	& \text{Quintet}(0, 2, 1, 4; 3), \
	  \text{Quintet}(0, 2, 3, 4; 1), \
	  \text{Quintet}(0, 3, 2, 4; 1)
	\bigg\}
	\Bigg].
	\label{eq_4_point}
\end{align*}
The evaluation of the above bound requires 9-dimensional numerical integration.
For such multi-dimensional integration, Monte Carlo or quasi-Monte Carlo methods are widely used, since the standard quadrature-based methods suffer from the curse of dimensionality \citep{caflisch1998monte}.

\subsection{Fast-Drone Behavior of the Upper Bounds}\label{sec_fast_drone_bound}

The constructions above are feasible TSPD tours and therefore give valid upper bounds for any $\alpha \geq 1$.
However, fixed local ring patterns need not characterize the optimal structure when the drone is very fast; longer alternating synchronization patterns may become relevant.
Nevertheless, the triangle-ring construction already yields a simple analytical bound in the limit $\alpha \to \infty$.

\begin{proposition}
\label{prop_fast_drone_upper}
For the speed-scaled Euclidean TSPD model,
\[
	\limsup_{\alpha \to \infty} \beta_\TSPD(\alpha) \leq \frac{1}{\sqrt{2}}\approx 0.7071.
\]
\end{proposition}

\begin{proof}
By the triangle-ring construction in Section~\ref{sec_triangle}, for any $h>0$,
\[
	\beta_\TSPD(\alpha) \leq \frac{1}{2h}\Eb[C_3(h,\alpha)],
\]
where
\[
	C_3(h,\alpha)
	=
	\max\left\{
	L_{02},
	\frac{1}{\alpha}(L_{01}+L_{12})
	\right\}.
\]
As $\alpha \to \infty$, $C_3(h,\alpha) \to L_{02}$.
Since $C_3(h,\alpha)\leq L_{02}+L_{01}+L_{12}$ for $\alpha\geq 1$, the dominated convergence theorem gives
\[
	\limsup_{\alpha\to\infty} \beta_\TSPD(\alpha)
	\leq
	\frac{1}{2h}\Eb[L_{02}].
\]
Since
\[
	L_{02}
	=
	\sqrt{(Z_1+Z_2)^2+h^4(U_0-U_2)^2},
\]
Jensen's inequality gives
\begin{align*}
	\Eb[L_{02}]
	&\leq
	\sqrt{\Eb[(Z_1+Z_2)^2]+h^4\Eb[(U_0-U_2)^2]} \\
	&=
	\sqrt{6+\frac{h^4}{6}}.
\end{align*}
Therefore,
\[
	\limsup_{\alpha\to\infty} \beta_\TSPD(\alpha)
	\leq
	\inf_{h>0}\frac{1}{2h}\sqrt{6+\frac{h^4}{6}}
	=
	\frac{1}{\sqrt{2}},
\]
where the infimum is attained at $h=\sqrt{6}$.
\end{proof}

Moreover, the upper-bound constructions reflect larger values of $\alpha$ through the optimization of the strip height $h$.
The optimizing strip height in the analytical triangle bound increases with $\alpha$ and converges to $\sqrt{6}$.
The Monte Carlo evaluations in Section~\ref{sec_results} show the same increasing pattern for the non-straight constructions over the reported speed ratios.

\section{Lower Bounds}
\label{sec_lower_bounds}

Considering the speed-scaled Euclidean distance model in Section~\ref{sec_special_case}, we turn our attention to lower bounds of $\beta_\TSPD$.
\subsection{Existing Lower Bounds}
We first notice this result:

\begin{lemma}[\citealt{agatz2018optimization}]
\label{lem_agatz}
For any given points $x_1, \ldots, x_n$, it holds that
\begin{equation}
\frac{1}{1+\alpha} \TSP(\{x_1, \ldots, x_n\}) \leq \TSPD(\{x_1, \ldots, x_n\})
\end{equation}
for any $\alpha \geq 1$.
\end{lemma}

This lemma immediately gives that
\begin{equation}\label{tspd_agatz_bound}
	\beta_\TSPD \geq \frac{\beta_\TSP}{1 + \alpha} \geq \frac{\beta}{1 + \alpha}
\end{equation}
for any lower bound $\beta$ for $\beta_\TSP$.
When $\beta = 0.6277$ and $\alpha = 2$, we obtain
\[
	\beta_\TSPD \geq 0.2092.
\]

Lower bounds from \citet{jones2026coordinated} for the sidekick model in which the truck may also serve customers remain valid for the TSPD after restricting their model to a single sidekick.
Their model permits continuous launch and rendezvous points along the truck path, and is therefore a relaxation of the TSPD.
After setting the truck speed to one and the sidekick speed to $\alpha$, their bounds can be written in our notation as follows:
\begin{lemma}[\citealt{jones2026coordinated}]
\label{lem_jones_bmed_explicit}
For the speed-scaled Euclidean TSPD model with $\alpha \geq 1$,
\begin{align}
	\beta_\TSPD
	&\geq
	\frac{\beta_\TSP}{1+\alpha},
	\label{eq_jones_tsp_explicit}\\
	\beta_\TSPD
	&\geq
	\sup_{d \in \mathbb{Z}_{\geq 2}}
	\frac{\sqrt{2}d(d!)^{1/(2d)}}
	{\sqrt{\pi e}(d+1)^{(2d+1)/(2d)}(d+\alpha)}.
	\label{eq_jones_bmed_explicit}
\end{align}
\end{lemma}
Note that \eqref{eq_jones_tsp_explicit} coincides with \eqref{tspd_agatz_bound}.
In addition, the bound in \eqref{eq_jones_bmed_explicit} is mainly relevant for large speed ratios or a large number of sidekicks; see \citet{jones2026coordinated}.
Although these bounds are valid, \eqref{tspd_agatz_bound}, equivalently \eqref{eq_jones_tsp_explicit}, does not consider the large-scale behavior of the TSPD, while \eqref{eq_jones_bmed_explicit} does not use its synchronization structure.
\subsection{Nearest-Neighbor Lower Bound}
\label{sec_nn_lower_bound}
We first derive a lower bound leveraging both features.

Consider $n$ points, $X_1, \ldots, X_n$, uniformly distributed on $[0,1]^2$, and suppose that a feasible TSPD solution is given.
Let $T$ be the set of customer nodes visited by the truck, including combined nodes, and let $D$ be the set of drone nodes.
Thus, $T$ and $D$ form a partition of the customer nodes.
Let the total distance traveled by the truck be $L_T$ and the total distance traveled by the drone \emph{alone} be $L_D$.
We let $L$ be the makespan of the TSPD.
For each customer $X_i$, let $Y_i$ be the sum of the distances from $X_i$ to its nearest and second-nearest other customers.
Using this, we present the following result:

\begin{lemma}\label{lem_tspd_nn_lower_bound}
For the speed-scaled Euclidean TSPD model with $\alpha \geq 1$, the following holds:
\begin{equation}
	\label{tspd_lower_bound}
	\beta_\TSPD \geq \frac{5}{4(\alpha+2)}.
\end{equation}
\end{lemma}
\begin{proof}
For any given TSPD solution,
\[
	L \geq \max\Big\{ L_T, \frac{1}{\alpha} L_D \Big\}.
\]
We first bound $L_T$ and $L_D$ separately.

For each truck-visited node $X_i \in T$, the truck route has two incident truck edges.
Since $Y_i$ is the sum of the distances from $X_i$ to its nearest and second-nearest other customers, the sum of these two incident edge lengths is at least $Y_i$.
Summing over all $X_i \in T$ and noting that each truck edge is counted twice gives
\begin{equation}
\label{eq_truck_nn_lower_bound}
	L_T \geq \frac{1}{2}\sum_{X_i \in T} Y_i.
\end{equation}

For each drone-served node $X_i \in D$, the drone travels from a launch node to $X_i$ and then from $X_i$ to a recovery node.
The launch and recovery nodes are distinct customer nodes under the ring definition, so the total length of these two drone edges is at least $Y_i$.
Hence,
\[
	L_D \geq \sum_{X_i \in D} Y_i.
\]

Combining the truck and drone bounds,
\[
	L
	\geq
	\max\Bigg\{
		\frac{1}{2}\sum_{X_i \in T}Y_i,
		\frac{1}{\alpha}\sum_{X_i \in D}Y_i
	\Bigg\}
	\geq
	\frac{
		\sum_{X_i \in T}Y_i
		+
		\sum_{X_i \in D}Y_i
	}{\alpha+2}
	= \frac{1}{\alpha+2}\sum_{i=1}^n Y_i.
\]
The second inequality follows from $\max\{a/2,b/\alpha\} \geq (a+b)/(\alpha+2)$ for any $a,b\geq 0$.
Taking expectations and using exchangeability,
\[
	\Eb L \geq \frac{1}{\alpha+2}\Eb\Bigg[\sum_{i=1}^n Y_i\Bigg]
	= \frac{n}{\alpha+2}\Eb Y_1.
\]
By Lemma~\ref{lem_nearest},
\[
	\Eb Y_1
	=
	\left(\frac{1}{2}+\frac{3}{4}\right)\frac{1}{\sqrt n}
	=
	\frac{5}{4\sqrt n}.
\]
Therefore,
\[
	\liminf_{n\to\infty}\frac{\Eb L}{\sqrt n}
	\geq
	\frac{5}{4(\alpha+2)},
\]
which yields \eqref{tspd_lower_bound}.
\end{proof}

The lower bound in \eqref{tspd_lower_bound} dominates the bound in \eqref{tspd_agatz_bound} whenever $\beta \leq 5/6$.
Indeed, the inequality
\[
	\frac{5}{4(\alpha+2)} \geq \frac{\beta}{1+\alpha}
\]
is equivalent to
\[
	\beta \leq \frac{5}{4}\frac{\alpha+1}{\alpha+2}.
\]
For all $\alpha \geq 1$, the right-hand side is at least $5/6$.
Since $\beta=0.6277 < 5/6$, the lower bound in \eqref{tspd_lower_bound} is stronger than the bound in \eqref{tspd_agatz_bound} for all $\alpha \geq 1$.
This also holds if we use the empirical value $\beta_\TSP \approx 0.71$.

\subsection{$k$-TSP Lower Bound}
However, the lower bound in \eqref{tspd_lower_bound}, like those in \eqref{tspd_agatz_bound} and \eqref{eq_jones_bmed_explicit}, vanishes as $\alpha \to \infty$.
In the fast-drone regime, we notice a natural connection to the $k$-TSP, which seeks a shortest path visiting $k$ of $n$ customers.
As the drone speed increases, its contribution to the makespan diminishes, and the makespan becomes increasingly governed by the truck path.
This observation yields a lower bound that remains positive as $\alpha \to \infty$.

\begin{lemma}\label{lem_tspd_ktsp_lower_bound}
For the speed-scaled Euclidean TSPD model with $\alpha \geq 1$, the following holds:
\begin{equation}
	\label{tspd_ktsp_lower_bound}
	\beta_\TSPD \geq \frac{1}{e\sqrt{2\pi}} \approx 0.1468.
\end{equation}
\end{lemma}
\begin{proof}
For $2 \leq k \leq n$, let $\kTSP(\{X_1,\ldots,X_n\},k)$ denote the minimum length of a path visiting $k$ points selected from $\{X_1,\ldots,X_n\}$.
Let $m=|T|$ be the number of customer nodes visited by the truck.
The number of rings equals the number of combined nodes, and each ring contains at most one drone-served node.
Therefore, $|D|\leq |T|$, and hence $m\geq \lceil n/2\rceil$.
Since the truck route is a cycle through $m$ customer nodes, it contains a path through $\lceil n/2\rceil$ customer nodes.
Therefore,
\[
	L \geq L_T \geq \kTSP(\{X_1,\ldots,X_n\},\lceil n/2\rceil).
\]

\citet{blanchard2024probabilistic} give
\[
	\Eb[\kTSP(\{X_1,\ldots,X_n\},k)]
	\geq
	\sqrt{\frac{2}{e^2\pi}}
	\frac{k-1}{n^{\frac{1}{2}(1+\frac{1}{k-1})}}
	\int_0^1
	\left(1-\frac{\epsilon^{2k-2}}{2\sqrt{\pi(k-1)}}\right)\,d\epsilon
\]
and hence
\[
	\Eb[\kTSP(\{X_1,\ldots,X_n\},k)]
	\geq
	\sqrt{\frac{2}{e^2\pi}}
	\frac{k-1}{n^{\frac{1}{2}(1+\frac{1}{k-1})}}
	\left(1-\frac{1}{2(2k-1)\sqrt{\pi(k-1)}}\right).
\]
Set $k=\lceil n/2\rceil$.
Dividing by $\sqrt n$ and using $L\geq \kTSP(\{X_1,\ldots,X_n\},\lceil n/2\rceil)$ gives
\[
\begin{aligned}
	\frac{\Eb L}{\sqrt n}
	&\geq
	\frac{\Eb[\kTSP(\{X_1,\ldots,X_n\},\lceil n/2\rceil)]}{\sqrt n} \\
	&\geq
	\sqrt{\frac{2}{e^2\pi}}
	\frac{\lceil n/2\rceil-1}{n^{1+\frac{1}{2(\lceil n/2\rceil-1)}}}
	\left(1-\frac{1}{2(2\lceil n/2\rceil-1)\sqrt{\pi(\lceil n/2\rceil-1)}}\right).
\end{aligned}
\]
As $n\to\infty$,
\[
\begin{gathered}
	\frac{\lceil n/2\rceil-1}{n}\longrightarrow\frac12,
	\qquad
	n^{-\frac{1}{2(\lceil n/2\rceil-1)}}
	=
	\exp\left(-\frac{\log n}{2(\lceil n/2\rceil-1)}\right)
	\longrightarrow1, \\
	1-\frac{1}{2(2\lceil n/2\rceil-1)\sqrt{\pi(\lceil n/2\rceil-1)}}\longrightarrow1.
\end{gathered}
\]
Therefore,
\[
	\liminf_{n\to\infty}\frac{\Eb L}{\sqrt n}
	\geq
	\sqrt{\frac{2}{e^2\pi}}\frac12
	=
	\frac{1}{e\sqrt{2\pi}},
\]
which yields \eqref{tspd_ktsp_lower_bound}.
\end{proof}

The lower bound in \eqref{tspd_ktsp_lower_bound} dominates the bound in \eqref{eq_jones_bmed_explicit}.
By the arithmetic--geometric mean inequality,
\[
	(d!)^{1/(2d)} \leq \sqrt{\frac{d+1}{2}}.
\]
Therefore, for every $d\geq 2$ and $\alpha\geq 1$,
\[
\begin{aligned}
	\frac{\sqrt{2}d(d!)^{1/(2d)}}
	{\sqrt{\pi e}(d+1)^{(2d+1)/(2d)}(d+\alpha)}
	&\leq
	\frac{d}{\sqrt{\pi e}(d+1)^{1/2+1/(2d)}(d+\alpha)} \\
	&\leq
	\frac{d}{\sqrt{\pi e}(d+1)^{3/2}}
	\leq
	\frac{1}{e\sqrt{2\pi}}.
\end{aligned}
\]
The last inequality follows because $(d+1)^3/d^2$ is nondecreasing for $d\geq 2$, and hence $(d+1)^3/d^2\geq 27/4>2e$.
Taking the supremum over $d\in\mathbb Z_{\geq 2}$ proves the claim.

Together with the nearest-neighbor bound in Section~\ref{sec_nn_lower_bound}, the $k$-TSP bound yields the following strengthened lower bound:

\begin{theorem}\label{thm_tspd_combined_lower_bound}
For the speed-scaled Euclidean TSPD model with $\alpha\geq 1$,
\begin{equation}
	\label{tspd_combined_lower_bound}
	\beta_\TSPD
	\geq
	\max\left\{
		\frac{5}{4(\alpha+2)},
		\frac{1}{e\sqrt{2\pi}}
	\right\}.
\end{equation}
\end{theorem}
\begin{proof}
The result follows directly from Lemmas~\ref{lem_tspd_nn_lower_bound} and~\ref{lem_tspd_ktsp_lower_bound}.
\end{proof}

The two terms in \eqref{tspd_combined_lower_bound} are equal at
\[
	\alpha
	=
	\frac{5e\sqrt{2\pi}}{4}-2
	\approx 6.5172.
\]
The nearest-neighbor bound is stronger below this value, whereas the $k$-TSP bound is stronger above it.

\section{Numerical Results}
\label{sec_results}

In this section, we numerically evaluate the bounds derived in Sections~\ref{sec_upper_bounds} and~\ref{sec_lower_bounds} for the speed-scaled Euclidean TSPD model in Section~\ref{sec_special_case}.
We also provide a comparison with the continuous model of \citet{carlsson2018coordinated}.

\subsection{Bounds and Empirical Estimates}

While the computation of the lower bounds and the analytical upper bounds is direct, tighter upper bounds require evaluations of multi-dimensional integrals.
Since these integrals represent expectations, Monte Carlo integration is a natural choice for numerical integration.
See \citet{robert1999monte}, for example, for the details of Monte Carlo integration.
We implemented a classical Monte Carlo integration in the Julia language and used the \texttt{Optim.jl} package \citep{mogensen2018optim} to find the best $h$ value for each simulation.

\begin{table}
\centering
\caption{
Numerical results for the speed-scaled Euclidean TSPD model.
(1) The top section presents the Monte Carlo evaluation of the upper bounds on $\beta_\TSPD$ with 20 million samples.
(2) The middle section presents values obtained by ICP initialized with the LKH heuristic;
the average of 100 random instances is reported for each $(n, \alpha)$.
(3) The bottom section reports lower bounds (LB).
Values following $\pm$ are 95\% confidence-interval half-widths.
All values are rounded.
}
\label{tbl_euclidean_combined}
\begingroup
\setlength{\tabcolsep}{4pt}
\begin{tabular}{l r rrrrr}
\toprule
&  & \multicolumn{5}{c}{$\TSPD / \sqrt{n}$} \\
  \cmidrule(lr){3-7}
&
&	$\alpha = 1.0$
&	$\alpha = 1.5$
&	$\alpha = 2.0$
&	$\alpha = 2.5$
&	$\alpha = 3.0$\\
\midrule
& straight only         & \meanstd{0.9211}{0.0002}  & \meanstd{0.9211}{0.0002}  & \meanstd{0.9211}{0.0002}  & \meanstd{0.9211}{0.0002}  & \meanstd{0.9211}{0.0002} \\
& triangle only         & \meanstd{0.9212}{0.0002}  & \meanstd{0.7424}{0.0001}  & \meanstd{0.6906}{0.0001}  & \meanstd{0.6671}{0.0001}  & \meanstd{0.6548}{0.0001} \\
& quartet only          & \meanstd{0.8318}{0.0001}  & \meanstd{0.6839}{0.0001}  & \meanstd{0.6568}{0.0001}  & \meanstd{0.6484}{0.0001}  & \meanstd{0.6452}{0.0001} \\
& 5-point patterns      & \meanstd{0.7606}{0.0001}  & \meanstd{0.6544}{0.0001}  & \meanstd{0.6131}{0.0001}  & \meanstd{0.5829}{0.0001}  & \meanstd{0.5616}{0.0001} \\
\midrule
\multirow{11}{*}{$n$}
&    20                 & \meanstd{0.6924}{0.0124}  & \meanstd{0.6084}{0.0118}  & \meanstd{0.5640}{0.0116}  & \meanstd{0.5346}{0.0115}  & \meanstd{0.5158}{0.0120} \\
&    30                 & \meanstd{0.6812}{0.0078}  & \meanstd{0.6012}{0.0069}  & \meanstd{0.5571}{0.0072}  & \meanstd{0.5305}{0.0074}  & \meanstd{0.5105}{0.0074} \\
&    50                 & \meanstd{0.6603}{0.0058}  & \meanstd{0.5821}{0.0054}  & \meanstd{0.5413}{0.0055}  & \meanstd{0.5128}{0.0060}  & \meanstd{0.4931}{0.0064} \\
&    75                 & \meanstd{0.6383}{0.0052}  & \meanstd{0.5637}{0.0048}  & \meanstd{0.5243}{0.0050}  & \meanstd{0.4977}{0.0051}  & \meanstd{0.4781}{0.0052} \\
&   100                 & \meanstd{0.6345}{0.0042}  & \meanstd{0.5607}{0.0040}  & \meanstd{0.5223}{0.0040}  & \meanstd{0.4977}{0.0044}  & \meanstd{0.4803}{0.0047} \\
&   200                 & \meanstd{0.6206}{0.0024}  & \meanstd{0.5499}{0.0023}  & \meanstd{0.5132}{0.0026}  & \meanstd{0.4893}{0.0028}  & \meanstd{0.4725}{0.0032} \\
&   500                 & \meanstd{0.6084}{0.0016}  & \meanstd{0.5397}{0.0014}  & \meanstd{0.5043}{0.0014}  & \meanstd{0.4818}{0.0016}  & \meanstd{0.4655}{0.0016} \\
&  1000                 & \meanstd{0.6021}{0.0011}  & \meanstd{0.5342}{0.0011}  & \meanstd{0.4993}{0.0013}  & \meanstd{0.4770}{0.0012}  & \meanstd{0.4609}{0.0013} \\
&  2000                 & \meanstd{0.5982}{0.0009}  & \meanstd{0.5316}{0.0009}  & \meanstd{0.4974}{0.0009}  & \meanstd{0.4756}{0.0010}  & \meanstd{0.4600}{0.0010} \\
&  5000                 & \meanstd{0.5946}{0.0005}  & \meanstd{0.5285}{0.0005}  & \meanstd{0.4946}{0.0005}  & \meanstd{0.4731}{0.0005}  & \meanstd{0.4576}{0.0005} \\
& 10000                 & \meanstd{0.5934}{0.0003}  & \meanstd{0.5277}{0.0003}  & \meanstd{0.4940}{0.0003}  & \meanstd{0.4726}{0.0004}  & \meanstd{0.4573}{0.0004} \\
\midrule
\multirow{4}{*}{LB}
& Ours                    & \meanonly{0.4167}  & \meanonly{0.3571}  & \meanonly{0.3125}  & \meanonly{0.2778}  & \meanonly{0.2500} \\
& \citeauthor{agatz2018optimization}$^*$
                          & \meanonly{0.3550}  & \meanonly{0.2840}  & \meanonly{0.2367}  & \meanonly{0.2029}  & \meanonly{0.1775} \\
& \citeauthor{agatz2018optimization}
                          & \meanonly{0.3139}  & \meanonly{0.2511}  & \meanonly{0.2092}  & \meanonly{0.1793}  & \meanonly{0.1569} \\
& \citeauthor{jones2026coordinated}
                          & \meanonly{0.0972}  & \meanonly{0.0863}  & \meanonly{0.0785}  & \meanonly{0.0726}  & \meanonly{0.0680} \\
\bottomrule
\end{tabular}
\endgroup
\end{table}

\begin{table}
\centering
\caption{Optimizing strip heights $h$ for the Monte Carlo evaluations of the non-straight Euclidean upper bounds.}
\label{tbl_euclidean_h}
\begin{tabular}{lccccc}
\toprule
& \multicolumn{5}{c}{Optimizing $h$} \\
\cmidrule(lr){2-6}
& $\alpha = 1.0$
& $\alpha = 1.5$
& $\alpha = 2.0$
& $\alpha = 2.5$
& $\alpha = 3.0$\\
\midrule
triangle only    & 1.7343 & 1.9587 & 2.1233 & 2.2249 & 2.2896 \\
quartet only     & 1.9047 & 2.1703 & 2.3146 & 2.3781 & 2.4071 \\
5-point patterns & 2.0828 & 2.3086 & 2.4512 & 2.5779 & 2.6882 \\
\bottomrule
\end{tabular}
\end{table}

\begin{figure}
\centering
\begin{subfigure}[t]{0.49\textwidth}
\centering
\begin{tikzpicture}
\begin{axis}[
	width=0.94\linewidth,
	height=0.68\linewidth,
	xmode=log,
	log basis x=10,
	xlabel={$\alpha$},
	ylabel={Lower bound on $\beta_\TSPD$},
	ylabel style={xshift=8pt},
	xmin=1,
	xmax=100,
	ymin=0,
	ymax=0.45,
	xtick={1,2,5,10,20,50,100},
	xticklabels={1,2,5,10,20,50,100},
	ytick={0,0.1,0.2,0.3,0.4},
	grid=major,
	legend cell align={left},
	legend columns=2,
	legend style={at={(0.5,1.03)}, anchor=south, draw=none, fill=none, font=\scriptsize},
	tick label style={font=\small},
	label style={font=\small}
]
\addplot[black, thick, mark=square*, mark repeat=2, mark options={fill=black, draw=black, line width=0pt}]
	coordinates {(1.0,0.416667) (1.5,0.357143) (2.0,0.312500) (2.5,0.277778) (3.0,0.250000) (5.0,0.178571) (6.5172,0.146763) (10.0,0.146763) (20.0,0.146763) (50.0,0.146763) (100.0,0.146763)};
\addlegendentry{Ours}
\addplot[black, thick, dashed, mark=*, mark repeat=2, mark options={fill=black, draw=black, line width=0pt}]
	coordinates {(1.0,0.355000) (1.5,0.284000) (2.0,0.236667) (2.5,0.202857) (3.0,0.177500) (5.0,0.118333) (6.5172,0.094450) (10.0,0.064545) (20.0,0.033810) (50.0,0.013922) (100.0,0.007030)};
\addlegendentry{Agatz et al.$^*$}
\addplot[black, thick, dashdotted, mark=triangle*, mark size=2.6pt, mark repeat=2, mark options={fill=black, draw=black, line width=0pt}]
	coordinates {(1.0,0.313850) (1.5,0.251080) (2.0,0.209233) (2.5,0.179343) (3.0,0.156925) (5.0,0.104617) (6.5172,0.083502) (10.0,0.057064) (20.0,0.029890) (50.0,0.012308) (100.0,0.006215)};
\addlegendentry{Agatz et al.}
\addplot[black, thick, dotted, mark=diamond*, mark repeat=2, mark options={fill=black, draw=black, line width=0pt}]
	coordinates {(1.0,0.097176) (1.5,0.086296) (2.0,0.078503) (2.5,0.072553) (3.0,0.068018) (5.0,0.055963) (6.5172,0.050297) (10.0,0.042024) (20.0,0.030915) (50.0,0.020159) (100.0,0.014436)};
\addlegendentry{Jones et al.}
\end{axis}
\end{tikzpicture}
\caption{Lower-bound comparison.}
\label{fig_euclidean_lower_bounds}
\end{subfigure}\hfill
\begin{subfigure}[t]{0.49\textwidth}
\centering
\begin{tikzpicture}
\begin{axis}[
	width=0.94\linewidth,
	height=0.68\linewidth,
	xlabel={$\alpha$},
	ylabel={$\TSPD/\sqrt{n}$},
	ylabel style={xshift=8pt},
	xmin=0.9,
	xmax=3.1,
	ymin=0.2,
	ymax=0.8,
	xtick={1.0,1.5,2.0,2.5,3.0},
	ytick={0.2,0.3,0.4,0.5,0.6,0.7,0.8},
	grid=major,
	legend cell align={left},
	legend columns=2,
	legend style={at={(0.5,1.03)}, anchor=south, draw=none, fill=none, font=\scriptsize},
	tick label style={font=\small},
	label style={font=\small}
]
\addplot[black, thick, mark=square*, mark options={fill=black, draw=black, line width=0pt}]
	coordinates {(1.0,0.7606) (1.5,0.6544) (2.0,0.6131) (2.5,0.5829) (3.0,0.5616)};
\addlegendentry{5-point upper bound}
\addplot[black, thick, dashed, mark=*, mark options={fill=black, draw=black, line width=0pt}]
	coordinates {(1.0,0.5934) (1.5,0.5277) (2.0,0.4940) (2.5,0.4726) (3.0,0.4573)};
\addlegendentry{ICP, $n=10000$}
\addplot[black, thick, dotted, mark=triangle*, mark size=2.6pt, mark options={fill=black, draw=black, line width=0pt}]
	coordinates {(1.0,0.4167) (1.5,0.3571) (2.0,0.3125) (2.5,0.2778) (3.0,0.2500)};
\addlegendentry{our lower bound}
\end{axis}
\end{tikzpicture}
\caption{Bounds and ICP estimate.}
\label{fig_euclidean_bounds_numerical}
\end{subfigure}
\caption{Numerical results for the speed-scaled Euclidean TSPD model. (a) Lower-bound comparison over a broader range of speed ratios. (b) Best upper bound, ICP estimate, and our lower bound. The starred Agatz et al. curve uses the empirical value $\beta_\TSP=0.71$.}
\label{fig_euclidean_bounds}
\end{figure}

Table~\ref{tbl_euclidean_combined} reports the computational results for the speed-scaled Euclidean TSPD model.
The top section displays the Monte Carlo estimates of the upper bounds for various ring patterns.
As expected, the 5-point patterns provide the best bounds, while the bounds improve as we use more complicated ring patterns.
Table~\ref{tbl_euclidean_h} reports the corresponding optimizing values of $h$ for the Monte Carlo evaluations.
For the non-straight constructions, the optimizing $h$ increases with $\alpha$.

The bottom section lists lower bounds.
Ours reports \eqref{tspd_combined_lower_bound}.
Since all speed ratios in Table~\ref{tbl_euclidean_combined} are below 6.5172, the reported values are determined by the nearest-neighbor term.
\citeauthor{agatz2018optimization}$^*$ reports \eqref{tspd_agatz_bound} with the empirical value $\beta_\TSP=0.71$, while \citeauthor{agatz2018optimization} reports the same bound with the valid lower bound $\beta=0.6277$.
\citeauthor{jones2026coordinated} reports \eqref{eq_jones_bmed_explicit}.
For \eqref{eq_jones_bmed_explicit}, we enumerated $d=2,\ldots,99$.
For $d \geq 100$, since $d! \leq d^d$, $(d+1)^{(2d+1)/(2d)} \geq d$, and $d+\alpha \geq d$,
\[
	\frac{\sqrt{2}d(d!)^{1/(2d)}}{\sqrt{\pi e}(d+1)^{(2d+1)/(2d)}(d+\alpha)}
	\leq
	\frac{\sqrt{2}d^{3/2}}{\sqrt{\pi e}d^2}
	=
	\frac{\sqrt{2/(\pi e)}}{\sqrt d}
	\leq 0.0484,
\]
which is below the reported values.
Thus, the supremum is attained in the enumerated range.
The maximizing values of $d$ for $\alpha=1.0,1.5,2.0,2.5,3.0$ are $2,3,4,5,5$, respectively.
As expected, our lower bound is stronger than the bounds from both \citet{agatz2018optimization} and \citet{jones2026coordinated}.
Figure~\ref{fig_euclidean_lower_bounds} compares these lower bounds over a broader range of speed ratios.

To compare the upper and lower bounds for TSPD, we obtained actual solutions of TSPD in the middle section of Table~\ref{tbl_euclidean_combined}.
Since we do not have an exact algorithm for large-scale TSPD, we used the iterative chainlet partitioning (ICP) algorithm of \citet{lee2025iterative} to obtain high-quality heuristic solutions.
The ICP algorithm starts with an initial TSP tour and iteratively improves it via local search and dynamic programming.
Compared to existing algorithms, ICP provides the best solutions for most benchmark instances.
We obtained initial tours by the LKH solver~\citep{helsgaun2000effective}.
We solved 100 randomly generated instances for each instance size $n$ and reported the average value.
Since ICP is a heuristic algorithm, the values in Table~\ref{tbl_euclidean_combined} are approximations to upper bounds.
For $\alpha=2.0$, these values stabilize near 0.49 as $n$ increases, leading us to conjecture that $\beta_\TSPD \approx 0.49$.
Figure~\ref{fig_euclidean_bounds_numerical} summarizes the best upper bound, the ICP estimate for $n=10000$, and our lower bound.

\subsection{Cost of Discreteness}

We assess the cost of restricting launch and recovery locations to customer nodes by comparing our empirical TSPD results with numerical results for the continuous model.
In addition to their analytical results, \citet{carlsson2018coordinated} obtain computational results using a heuristic algorithm for a model in which the truck may also serve customer locations, while drone launch and rendezvous points are still chosen continuously along the truck path.
Table~\ref{tbl_discreteness_cost} reports the relative reduction in service time obtained by using a drone, measured against the corresponding truck-only TSP service time.
The \emph{Continuous} row is converted from the computational results of \citet{carlsson2018coordinated}, and the \emph{Discrete} row is computed by ICP initialized with LKH.
Both rows present averages over 200 random instances with $n=500$.

\begin{table}
\centering
\caption{
Relative service-time reduction in the continuous and discrete models.
The \emph{Continuous} row is converted from \citet{carlsson2018coordinated}, and the \emph{Discrete} row is obtained by ICP initialized with the LKH heuristic.
The last row reports the gap between the two reductions.
}
\label{tbl_discreteness_cost}
\begin{tabular}{l ccccc}
\toprule
& \multicolumn{5}{c}{$1-\TSPD/\TSP$} \\
\cmidrule(lr){2-6}
& $\alpha=1.5$
& $\alpha=2.0$
& $\alpha=2.5$
& $\alpha=3.0$
& $\alpha=5.0$ \\
\midrule
\emph{Continuous} & 0.3996 & 0.4343 & 0.4302 & 0.4339 & 0.5482 \\
\emph{Discrete} & 0.2701 & 0.3180 & 0.3490 & 0.3710 & 0.4187 \\
\midrule
Gap & 0.1295 & 0.1163 & 0.0813 & 0.0629 & 0.1295 \\
\bottomrule
\end{tabular}
\end{table}

Although the continuous model provides a larger reduction, the difference is small for the speed ratios considered here.
Indeed, the difference decreases as $\alpha$ increases from 1.5 to 3.0.
These results suggest that restricting launch and recovery to customer nodes does not appear to impose a large empirical penalty for moderate drone speeds.
However, the difference increases at $\alpha=5.0$, indicating that continuous launch and rendezvous may become more valuable when the drone is very fast.

\section{The Rectilinear-Euclidean Mixed TSPD Model}
\label{sec_mixed}

It is possible to extend our result to the Rectilinear-Euclidean mixed model in Section~\ref{sec_mixed_model}.
The existence of the asymptotic constant follows directly from Theorem~\ref{thm_tspd_limit}, since the rectilinear truck metric and the speed-scaled Euclidean drone metric satisfy its assumptions.
In this section, $\beta_\TSPD$ denotes the corresponding constant.

For the upper bound, the same strip constructions can be applied.
For the homogeneous ring constructions in Sections~\ref{sec_straight}--\ref{sec_quartet} and the five-point patterns in Section~\ref{sec_five_point_patterns}, we replace every truck edge length by the rectilinear distance, while drone edge lengths remain speed-scaled Euclidean distances.
The reduction based on Theorem~\ref{thm_no_straight_ring} also remains valid because the mixed model satisfies
\[
	\frac{1}{\alpha}\|x-y\|_2 \leq \|x-y\|_1.
\]
Therefore, the straight-ring-containing alternatives in the four- and five-point pattern enumerations can be omitted as before.
The numerical evaluation differs only in the ring-cost formulas.
In particular, the simplified quartet expression in Lemma~\ref{lem_c4} is specific to the common-metric case; for the mixed model, we evaluate the candidate quartet costs directly.

For the lower bound, the $k$-TSP argument can be strengthened under the rectilinear truck metric.

\begin{lemma}\label{lem_mixed_ktsp_lower_bound}
For the Rectilinear-Euclidean mixed TSPD model with $\alpha\geq 1$,
\begin{equation}
	\label{mixed_tspd_ktsp_lower_bound}
	\beta_\TSPD \geq \frac{1}{2e} \approx 0.1839.
\end{equation}
\end{lemma}
\begin{proof}
For $2\leq k\leq n$, let $\kTSP^{(1)}(\{X_1,\ldots,X_n\},k)$ denote the minimum rectilinear length of a path visiting $k$ points selected from $\{X_1,\ldots,X_n\}$.
As in the proof of Lemma~\ref{lem_tspd_ktsp_lower_bound}, the truck visits at least $\lceil n/2\rceil$ customer nodes, and its cycle contains a path through $\lceil n/2\rceil$ customer nodes.
Therefore,
\[
	L \geq L_T \geq \kTSP^{(1)}(\{X_1,\ldots,X_n\},\lceil n/2\rceil).
\]

Since an $\ell_1$-ball of radius $r$ has area $2r^2$, applying the argument of \citet{blanchard2024probabilistic} to the rectilinear metric gives
\[
	\Eb[\kTSP^{(1)}(\{X_1,\ldots,X_n\},k)]
	\geq
	\frac{1}{e}
	\frac{k-1}{n^{\frac{1}{2}(1+\frac{1}{k-1})}}
	\left(1-\frac{1}{2(2k-1)\sqrt{\pi(k-1)}}\right).
\]
Setting $k=\lceil n/2\rceil$ and using the limits in the proof of Lemma~\ref{lem_tspd_ktsp_lower_bound},
\[
	\liminf_{n\to\infty}\frac{\Eb L}{\sqrt n}
	\geq
	\frac{1}{e}\frac12
	=
	\frac{1}{2e},
\]
which yields \eqref{mixed_tspd_ktsp_lower_bound}.
\end{proof}

We thus obtain the following strengthened lower bound for the mixed model:

\begin{theorem}\label{thm_mixed_combined_lower_bound}
For the Rectilinear-Euclidean mixed TSPD model with $\alpha\geq 1$,
\begin{equation}
	\label{mixed_tspd_combined_lower_bound}
	\beta_\TSPD
	\geq
	\max\left\{
		\frac{5}{4(\alpha+2)},
		\frac{1}{2e}
	\right\}.
\end{equation}
\end{theorem}
\begin{proof}
Since every rectilinear truck edge dominates the corresponding Euclidean edge, the truck bound in \eqref{eq_truck_nn_lower_bound} still holds.
The drone bound is unchanged because the drone metric is Euclidean.
Thus, the proof of Lemma~\ref{lem_tspd_nn_lower_bound} yields the first term.
The second term follows from Lemma~\ref{lem_mixed_ktsp_lower_bound}.
\end{proof}

The two terms in \eqref{mixed_tspd_combined_lower_bound} are equal at
\[
	\alpha
	=
	\frac{5e}{2}-2
	\approx 4.7957.
\]
The nearest-neighbor bound is stronger below this value, whereas the rectilinear $k$-TSP bound is stronger above it.

\begin{table}
\centering
\caption{
Numerical results for the Rectilinear-Euclidean mixed TSPD model.
(1) The top section presents the Monte Carlo evaluation of the upper bounds on $\beta_\TSPD$ with 20 million samples.
(2) The middle section presents values obtained by ICP initialized with the LKH heuristic;
the average of 100 random instances is reported for each $(n, \alpha)$.
(3) The bottom section reports the lower bound (LB).
Values following $\pm$ are 95\% confidence-interval half-widths.
All values are rounded.
}
\label{tbl_rectilinear_combined}
\begingroup
\setlength{\tabcolsep}{4pt}
\begin{tabular}{l r rrrrr}
\toprule
&  & \multicolumn{5}{c}{$\TSPD / \sqrt{n}$} \\
  \cmidrule(lr){3-7}
&
&	$\alpha = 1.0$
&	$\alpha = 1.5$
&	$\alpha = 2.0$
&	$\alpha = 2.5$
&	$\alpha = 3.0$\\
\midrule
& straight only         & \meanstd{1.1547}{0.0003}  & \meanstd{1.1547}{0.0003}  & \meanstd{1.1547}{0.0003}  & \meanstd{1.1547}{0.0003}  & \meanstd{1.1547}{0.0003} \\
& triangle only         & \meanstd{1.0069}{0.0002}  & \meanstd{0.8926}{0.0002}  & \meanstd{0.8552}{0.0002}  & \meanstd{0.8383}{0.0002}  & \meanstd{0.8297}{0.0002} \\
& quartet only          & \meanstd{0.9097}{0.0001}  & \meanstd{0.8385}{0.0001}  & \meanstd{0.8237}{0.0001}  & \meanstd{0.8193}{0.0001}  & \meanstd{0.8177}{0.0001} \\
& 5-point patterns      & \meanstd{0.8626}{0.0001}  & \meanstd{0.7854}{0.0001}  & \meanstd{0.7380}{0.0001}  & \meanstd{0.7085}{0.0001}  & \meanstd{0.6910}{0.0001} \\
\midrule
\multirow{11}{*}{$n$}
&    20                 & \meanstd{0.7962}{0.0147}  & \meanstd{0.7099}{0.0137}  & \meanstd{0.6690}{0.0145}  & \meanstd{0.6414}{0.0150}  & \meanstd{0.6140}{0.0153} \\
&    30                 & \meanstd{0.7938}{0.0089}  & \meanstd{0.7060}{0.0082}  & \meanstd{0.6590}{0.0085}  & \meanstd{0.6280}{0.0084}  & \meanstd{0.6057}{0.0091} \\
&    50                 & \meanstd{0.7620}{0.0066}  & \meanstd{0.6801}{0.0067}  & \meanstd{0.6354}{0.0070}  & \meanstd{0.6049}{0.0078}  & \meanstd{0.5849}{0.0081} \\
&    75                 & \meanstd{0.7407}{0.0060}  & \meanstd{0.6623}{0.0059}  & \meanstd{0.6197}{0.0060}  & \meanstd{0.5907}{0.0060}  & \meanstd{0.5710}{0.0066} \\
&   100                 & \meanstd{0.7344}{0.0049}  & \meanstd{0.6574}{0.0052}  & \meanstd{0.6170}{0.0059}  & \meanstd{0.5907}{0.0061}  & \meanstd{0.5730}{0.0064} \\
&   200                 & \meanstd{0.7206}{0.0030}  & \meanstd{0.6454}{0.0033}  & \meanstd{0.6059}{0.0037}  & \meanstd{0.5782}{0.0043}  & \meanstd{0.5608}{0.0044} \\
&   500                 & \meanstd{0.7074}{0.0019}  & \meanstd{0.6361}{0.0018}  & \meanstd{0.5979}{0.0019}  & \meanstd{0.5741}{0.0020}  & \meanstd{0.5565}{0.0020} \\
&  1000                 & \meanstd{0.7011}{0.0013}  & \meanstd{0.6299}{0.0014}  & \meanstd{0.5925}{0.0015}  & \meanstd{0.5681}{0.0016}  & \meanstd{0.5510}{0.0015} \\
&  2000                 & \meanstd{0.6974}{0.0010}  & \meanstd{0.6274}{0.0011}  & \meanstd{0.5908}{0.0012}  & \meanstd{0.5670}{0.0013}  & \meanstd{0.5502}{0.0012} \\
&  5000                 & \meanstd{0.6929}{0.0006}  & \meanstd{0.6242}{0.0006}  & \meanstd{0.5880}{0.0007}  & \meanstd{0.5643}{0.0007}  & \meanstd{0.5479}{0.0007} \\
& 10000                 & \meanstd{0.6920}{0.0004}  & \meanstd{0.6233}{0.0004}  & \meanstd{0.5874}{0.0005}  & \meanstd{0.5642}{0.0005}  & \meanstd{0.5479}{0.0005} \\
\midrule
LB
& Ours                    & \meanonly{0.4167}  & \meanonly{0.3571}  & \meanonly{0.3125}  & \meanonly{0.2778}  & \meanonly{0.2500} \\
\bottomrule
\end{tabular}
\endgroup
\end{table}

\begin{figure}
\centering
\begin{tikzpicture}
\begin{axis}[
	width=0.78\textwidth,
	height=0.46\textwidth,
	xlabel={$\alpha$},
	ylabel={$\TSPD/\sqrt{n}$},
	xmin=0.9,
	xmax=3.1,
	ymin=0.2,
	ymax=0.9,
	xtick={1.0,1.5,2.0,2.5,3.0},
	ytick={0.2,0.3,0.4,0.5,0.6,0.7,0.8,0.9},
	grid=major,
	legend cell align={left},
	legend columns=3,
	legend style={at={(0.5,1.03)}, anchor=south, draw=none, fill=none, font=\small},
	tick label style={font=\small},
	label style={font=\small}
]
\addplot[black, thick, mark=square*, mark options={fill=black, draw=black, line width=0pt}]
	coordinates {(1.0,0.8626) (1.5,0.7854) (2.0,0.7380) (2.5,0.7085) (3.0,0.6910)};
\addlegendentry{5-point upper bound}
\addplot[black, thick, dashed, mark=*, mark options={fill=black, draw=black, line width=0pt}]
	coordinates {(1.0,0.6920) (1.5,0.6233) (2.0,0.5874) (2.5,0.5642) (3.0,0.5479)};
\addlegendentry{ICP, $n=10000$}
\addplot[black, thick, dotted, mark=triangle*, mark size=2.6pt, mark options={fill=black, draw=black, line width=0pt}]
	coordinates {(1.0,0.4167) (1.5,0.3571) (2.0,0.3125) (2.5,0.2778) (3.0,0.2500)};
\addlegendentry{our lower bound}
\end{axis}
\end{tikzpicture}
\caption{Best upper bound, ICP estimate, and our lower bound for the Rectilinear-Euclidean mixed TSPD model.}
\label{fig_mixed_bounds}
\end{figure}

Table~\ref{tbl_rectilinear_combined} reports the computational results for the Rectilinear-Euclidean mixed TSPD model.
The top section displays the Monte Carlo estimates of the upper bounds for various ring patterns, using the mixed ring-cost formulas described above.
The middle section reports the values obtained by ICP initialized with the LKH heuristic.
Since ICP is a heuristic algorithm, these values are approximations to upper bounds.
The bottom section reports our lower bound in \eqref{mixed_tspd_combined_lower_bound}.
For the speed ratios shown, the reported values are determined by the nearest-neighbor term.
When $\alpha=2.0$, the value decreases to 0.5874 at $n=10000$, suggesting that the mixed-model constant is perhaps close to 0.58.
Figure~\ref{fig_mixed_bounds} summarizes the best upper bound, the ICP estimate for $n=10000$, and our lower bound.

\section{Concluding Remarks}
\label{sec_conclusion}

We established a Beardwood--Halton--Hammersley-type asymptotic limit for the Traveling Salesman Problem with Drone (TSPD). 
For the speed-scaled Euclidean TSPD model, our analysis establishes the limit through nonmonotone subadditive Euclidean functional arguments and derives upper bounds via structured ring-based tour constructions.
A key insight is that optimal TSPD solutions can be assumed to exclude truck-drone co-travel segments, which significantly reduces the computational burden and enables efficient Monte Carlo evaluations.
We also developed lower bounds that link cooperative routing to nearest-neighbor distance distributions and the $k$-traveling salesman problem.
The resulting lower bound improves existing bounds and remains positive in the fast-drone regime.
The proposed methodological framework, which also applies to the Rectilinear--Euclidean mixed TSPD model, would provide a general template for analyzing the asymptotic behavior of a broader class of cooperative and multi-modal routing problems with synchronization constraints and heterogeneous distance metrics.

Numerical experiments compare the proposed bounds with empirical values obtained by ICP.
For the speed-scaled Euclidean model, our results suggest that the TSPD constant is close to 0.49 when the drone travels twice as fast as the truck.
For practical strategic planning, this result supports the use of the following continuous approximation formula for the optimal TSPD makespan serving $n$ customers in a service region of area $A$:
\(
	\TSPD \approx 0.49 \sqrt{nA}.
\)
Comparison with the continuous model indicates that restricting launch and recovery to customer nodes imposes a modest empirical penalty for moderate drone speeds.
This suggests that continuously chosen launch and recovery locations may not be essential in practice.

\section*{Acknowledgements}
This work was supported by the National Research Foundation of Korea (NRF) grant funded by the Korean government (MSIT) (RS-2023-00259550).

\bibliographystyle{dcu}
\bibliography{library}

\end{document}